\documentclass[11pt]{amsart}
\usepackage[margin=1.25in]{geometry}
\usepackage{algorithm, algpseudocode}
\usepackage{amsthm, amsmath, amssymb, mathtools, appendix, graphicx}
\usepackage{amsaddr}
\usepackage{hyperref}
\usepackage{cleveref}

\newtheorem{THM}{Theorem}
\newtheorem{LMA}[THM]{Lemma}

\newtheorem{CORO}[THM]{Corollary}

\newcommand\ux[3]{{#1}_{#2}^{(#3)}}
\newcommand{\RR}{\mathbb{R}}
\newcommand\ov[1]{\overline{#1}}
\newcommand{\lb}{\lbrack}
\newcommand{\rb}{\rbrack}
\newcommand{\la}{\langle}
\newcommand{\ra}{\rangle}
\newcommand\opn[1]{\operatorname{#1}}
\newcommand\wh[1]{{\widehat{#1}}}
\newcommand\wt[1]{{\widetilde{#1}}}
\newcommand{\Sym}{\Sigma}

\newcommand\half{\frac{1}{2}}
\numberwithin{THM}{section}
\numberwithin{equation}{section}

\DeclarePairedDelimiter{\floor}{\lfloor}{\rfloor}
\newcommand\imgscale{0.84}

\begin{document}
\title{Block BFGS Methods}
\author{Wenbo Gao and Donald Goldfarb$^\dagger$}
\address{Department of Industrial Engineering and Operations Research, Columbia University}
\thanks{$^\dagger$Research of this author was supported in part by NSF Grant CCF-1527809.}
\email{\texttt{wg2279@columbia.edu, goldfarb@columbia.edu}}
\subjclass[2010]{90C53, 90C26} 
\date{November 8, 2017.}
\maketitle
\begin{abstract}
We introduce a quasi-Newton method with block updates called \emph{Block BFGS}. We show that this method, performed with inexact Armijo-Wolfe line searches, converges globally and $q$-stage superlinearly under the same convexity assumptions as BFGS. We also show that Block BFGS is globally convergent to a stationary point when applied to non-convex functions with bounded Hessian, and discuss other modifications for non-convex minimization. Numerical experiments comparing Block BFGS, BFGS and gradient descent are presented.
\end{abstract}

\section{Introduction}
The classical BFGS method is perhaps the best known \emph{quasi-Newton method} for minimizing an unconstrained function $f(x)$. These methods iteratively proceed along search directions $d_k =  -B_k^{-1}\nabla f(x_k)$, where $B_k$ is an approximation to the Hessian $\nabla^2f(x_k)$ at the current iterate $x_k$. Quasi-Newton methods differ primarily in the manner in which they update the approximation $B_k$. The BFGS method constructs an update $B_{k+1}$ that is the nearest matrix to $B_k$ (in a variable metric) satisfying the \emph{secant equation} $B_{k+1}(x_{k+1} - x_k) = \nabla f(x_{k+1}) - \nabla f(x_k)$ \cite{G1970}. This can be interpreted as modifying $B_k$ to act like $\nabla^2 f(x)$ along the step $x_{k+1} - x_k$, so that successive updates induce $B_k$ to resemble $\nabla^2 f(x)$ along the search directions.

A natural extension of the classical BFGS method is to incorporate information about $\nabla^2 f(x)$ along \emph{multiple} directions in each update. This further improves the accuracy of the local Hessian approximation, allowing one to obtain better search directions. Early work in this area includes the development by Schnabel \cite{SCH} of quasi-Newton methods that satisfy multiple (say, $q$) secant equations $B_{k+1}\ux{s}{k}{i} = \nabla f(x_{k+1}) - \nabla f(x_{k+1} - \ux{s}{k}{i})$ for directions $\ux{s}{k}{1},\ldots,\ux{s}{k}{q}$. This approach has the disadvantage that the resulting update is generally not symmetric, and considerable modifications are required to ensure $B_k$ remains positive definite. Consequently, there appears to have been little interest in quasi-Newton methods with block updates in the years following Schnabel's initial report.

More recently, a stochastic quasi-Newton method with block updates was introduced by Gower, Goldfarb, and Richt\'{a}rik \cite{GGR}. Their approach constructs an update which satisfies \emph{sketching equations} of the form
$$B_{k+1}\ux{s}{k}{i} = \nabla^2 f(x_{k+1}) \ux{s}{k}{i}$$
for multiple directions $\ux{s}{k}{i}$. By using sketching equations instead of secant equations, the update is guaranteed to remain symmetric, and in the case where $f(x)$ is convex, positive definite. The sketching equations can be thought of as `tangent' equations that require $B_{k+1}$ to incorporate information about the Hessian $\nabla^2 f(x_{k+1})$ at the most recent point $x_{k+1}$, as opposed to information about the average of $\nabla^2 f(x)$ between two points, i.e, along a secant. Consequently, in terms of the information used, the block updating formula is Newton-like rather than secant-like. A Hessian-vector product $\nabla^2 f(x_{k+1})\ux{s}{k}{i}$ can generally be computed much faster than the full Hessian $\nabla^2 f(x_{k+1})$, and the operation of computing $\nabla^2 f(x_{k+1})\ux{s}{k}{i}$ for multiple directions $\ux{s}{k}{1},\ldots,\ux{s}{k}{q}$ can be done in parallel.

Computing the Hessian-vector products $\nabla^2 f(x_{k+1}) \ux{s}{k}{i}$, referred to as \emph{Hessian actions}, involves additional work beyond that of classical BFGS updates, where the gradients can be reused to compute $\nabla f(x_{k+1}) - \nabla f(x_k)$. However, the increased cost of block updates may be justified in order to obtain better search directions, for the same reason that Newton's method often outperforms gradient descent: the greater cost per iteration is compensated by convergence in fewer iterations, in regions where the curvature can be used effectively. Our numerical experiments in Section 7 explored this trade-off, and we found that using block updates did result in performance gains on many problems. 

Other experiments indicate that quasi-Newton methods using Hessian actions and block updates are promising for empirical risk minimization problems arising from machine learning. Byrd, Hansen, Nocedal, and Yuan \cite{BHNS} proposed a stochastic limited-memory algorithm \emph{Stochastic Quasi-Newton} (SQN), in which the secant equation is replaced by a sub-sampled sketching equation $B_{k+1}s_k = \wh{\nabla}^2 f(x_{k+1}) s_k$ (here $\wh{\nabla}^2f(x)$ denotes a sub-sampled Hessian). The authors \cite{BHNS} remark that using the sub-sampled Hessian action avoids harmful effects from gradient differencing in the stochastic setting. In \cite{GGR}, a stochastic limited-memory method \emph{Stochastic Block L-BFGS}, using block updates, outperformed other state-of-the-art methods when applied to large-scale logistic regression problems.

In this paper, we introduce a deterministic quasi-Newton method \emph{Block BFGS}. The key feature of Block BFGS is the inclusion of information about $\nabla^2 f(x)$ along multiple directions, by enforcing that $B_{k+1}$ satisfies the sketching equations for a subset of previous search directions. We show that this method, performed with inexact Armijo-Wolfe line searches, has the same convergence properties as the classical BFGS method. Namely, if $f$ is twice differentiable, convex, and bounded below, and the gradient of $f$ is Lipschitz continuous, then Block BFGS converges. If, in addition, $f$ is strongly convex and the Hessian of $f$ is Lipschitz continuous, then Block BFGS achieves $Q$-superlinear convergence. Note that we use a slightly modified notion of $Q$-superlinear convergence: we prove that the sequence of quotients $ \|\ux{x}{k}{i+1} - x_\ast\| / \|\ux{x}{k}{i} - x_\ast\|$, with possibly a small number of terms removed, converges to 0. The precise statement of this result is given in \Cref{SUPERLINEAR}. We also note that our convergence results can easily be extended to block versions of the restricted Broyden class of quasi-Newton methods as in \cite{BNY}.

These results fill a gap in the theory of quasi-Newton methods, as updates based on the Hessian action have previously only been used within limited-memory methods, for which the analysis is significantly simpler. Because of its limited-memory nature, the Stochastic Block L-BFGS method in \cite{GGR} is only proved to be $R$-linearly convergent (in expectation, when using a fixed step size). For this method, as is the case for the deterministic L-BFGS method \cite{LBFGS}, the convergence rate that is proved is worse than the rate for gradient descent (GD), even though in practice, L-BFGS almost always converges far more rapidly than GD. We believe that our proof of the $Q$-superlinear convergence of Block BFGS in this paper provides a rationale for the superior performance of the Stochastic Block L-BFGS method, and behavior of deterministic limited-memory Block BFGS methods as well.

Block BFGS can also be applied to non-convex functions. We show that if $f$ has bounded Hessian, then Block BFGS converges to a stationary point of $f$. Modified forms of the classical BFGS method also have natural extensions to block updates, so modified block quasi-Newton methods are applicable in the non-convex setting.

The paper is organized as follows. \Cref{sec:PRELIM} contains preliminaries and describes Armijo-Wolfe inexact line searches. In \Cref{sec:ALGOS}, we formally define the Block BFGS method and several variants. In \Cref{sec:CONVERGENT} and \Cref{sec:SUPERLINEAR} respectively, we show that Block BFGS converges, and converges superlinearly, for $f$ satisfying appropriate conditions. In \Cref{sec:NCONV}, we show that Block BFGS converges for suitable non-convex functions, and describe several other modifications to adapt Block BFGS for non-convex optimization. In \Cref{sec:NUMER}, we present the results of numerical experiments for several classes of convex and non-convex problems.

\section{Preliminaries}\label{sec:PRELIM}
The following notation will be used. The objective function of $n$ variables is denoted by $f: \RR^n \rightarrow \RR$. We write $g(x)$ for the gradient $\nabla f(x)$ and $G(x)$ for the Hessian $\nabla^2 f(x)$. For a sequence $\{x_k\}$, $f_k = f(x_k)$ and $g_k = g(x_k)$. However, we deliberately use $G_k = G(x_{k+1})$ to simplify the update formula.

The norm $\|\cdot\|$ denotes the $L_2$ norm, or for matrices, the $L_2$ operator norm. The Frobenius norm will be explicitly indicated as $\|\cdot\|_F$. Angle brackets $\la \cdot, \cdot \ra$ denote the standard inner product $\la x,y \ra = y^Tx$ and the trace inner product $\la X,Y \ra = \opn{Tr}(Y^TX)$. We use either notation $\la x,y \ra$ or $y^Tx$ as is convenient. The symbol $\Sym^n$ denotes the space of $n \times n$ symmetric matrices, and $\preceq$ denotes the L\"{o}wner partial order; hence $A \succ 0$ means $A$ is positive definite.

An $L\Sigma L^T$ decomposition is a factorization of a positive definite matrix into a product $L \Sigma L^T$, where $L$ is lower triangular with ones on the diagonal, and $\Sigma = \opn{Diag}(\sigma_1^2,\ldots,\sigma_n^2)$. This is commonly called an $LDL^T$ decomposition in the literature, but we write $\Sigma$ in place of $D$ as we use $D$ to denote a matrix whose columns are previous search directions.

In the pseudocode for our algorithm, \texttt{size}$(A,1)$ and \texttt{size}$(A,2)$ denote the number of rows and columns of a matrix $A$ respectively. The $ij$-entry of a matrix $A$ will be denoted by $A_{ij}$. We use $\opn{Col}(A)$ to denote the linear space spanned by the columns of $A$. By convention, a sum over an empty index set is equal to 0.

Our inexact line search selects step sizes $\lambda_k$ satisfying the \emph{Armijo-Wolfe} conditions: for parameters $\alpha, \beta$ with $0 < \alpha < \half$ and $\alpha < \beta < 1$, the step satisfies
\begin{align}
	\label{armijo} f(x_k + \lambda_k d_k) &\leq f(x_k) + \alpha \lambda_k \la g_k, d_k\ra \\
	\label{wolfe} \la g(x_k+ \lambda_kd_k), d_k \ra &\geq \beta \la g_k, d_k\ra.
\end{align}
Furthermore, our line search always selects $\lambda_k = 1$ whenever this step size is admissible. This is important in the analysis of superlinear convergence in \Cref{sec:SUPERLINEAR}.

\section{Block quasi-Newton Methods}\label{sec:ALGOS}

In this section, we introduce \emph{Block BFGS}, a quasi-Newton method with block updates, and several variants. 

\subsection{Block BFGS}

\begin{algorithm}
	\caption{Block BFGS}
	\label{alg:BLOCK}
	\begin{algorithmic}[1]
		\Statex \textbf{input}: $\ux{x}{1}{1}, B_1, q$
		\For{$k = 1,2,3\ldots$}
		\For{$i = 1,\ldots,q$}
		\State $\ux{d}{k}{i} \leftarrow -B_k^{-1}\ux{g}{k}{i}$
		\State $\ux{\lambda}{k}{i} \leftarrow \textsc{linesearch}(\ux{x}{k}{i}, \ux{d}{k}{i})$
		\State $\ux{s}{k}{i} \leftarrow \ux{\lambda}{k}{i}\ux{d}{k}{i}$
		\State $\ux{x}{k}{i+1} \leftarrow \ux{x}{k}{i} + \ux{s}{k}{i}$
		\EndFor
		\State $G_k \leftarrow G(\ux{x}{k}{q+1})$
		\State $S_k \leftarrow \lb \ux{s}{k}{1} \ldots \ux{s}{k}{q}\rb$
		\State $D_k \leftarrow \textsc{filtersteps}(S_k, G_k)$
		\If{$D_k$ is not empty}
		\State $B_{k+1} \leftarrow B_k - B_kD_k(D_k^TB_kD_k)^{-1}D_k^TB_k + G_kD_k(D_k^TG_kD_k)^{-1}D_k^TG_k$\;
		\Else 
		\State $B_{k+1} \leftarrow B_k$
		\EndIf
		\State $\ux{x}{k+1}{1} \leftarrow \ux{x}{k}{q+1}$\;
		\EndFor
	\end{algorithmic}
\end{algorithm}

\begin{algorithm}
	\caption{\textsc{filtersteps}}
	\label{alg:FILTER}
	\begin{algorithmic}[1]
		\Statex \textbf{input}: $S_k, G_k$ \hspace{2em} \textbf{output}: $D_k$ \hspace{2em} \textbf{parameters}: threshold $\tau > 0$
		\State Initialize $D_k \leftarrow S_k, i \leftarrow 1$
		\While{$i \leq \texttt{size}(D_k,2)$}
		\State $\sigma_i^2 \leftarrow \lb D_k^TG_kD_k \rb_{ii} - \sum_{j=1}^{i-1} L_{ij}^2\Sigma_{jj}$
		\State $s_i \leftarrow $ column $i$ of $D_k$
		\If{ $\sigma_i^2 \geq \tau \|s_i\|^2$}
		\State $\Sigma_{ii} \leftarrow \sigma_i^2$
		\State $L_{ii} \leftarrow 1$
		\For{$j = i+1,\ldots, \texttt{size}(D_k,2)$}
		\State $L_{ji} \leftarrow\ \frac{1}{\Sigma_{ii}}( \lb D_k^TG_kD_k \rb_{ji} - \sum_{k=1}^{i-1} L_{ik}L_{jk}\Sigma_{kk})$
		\EndFor
		\State $i \leftarrow i+1$
		\Else
		\State Delete column $i$ from $D_k$ and row $i$ from $L$\;
		\EndIf
		\EndWhile
	\end{algorithmic}
\end{algorithm}

Block BFGS (Algorithm~\ref{alg:BLOCK}) takes $q$ steps in each block, using a fixed Hessian approximation $B_k$. We may also take a varying number of steps, bounded above by $q$, but we assume every block contains $q$ steps to simplify the presentation. We use a subscript $k$ for the block index, and superscripts $(i)$ for the steps within each block. The $k$-th block contains the iterates $\ux{x}{k}{1}, \ldots, \ux{x}{k}{q+1}$, and $\ux{x}{k+1}{1} = \ux{x}{k}{q+1}$. At each point $\ux{x}{k}{i}$, the step direction is $\ux{d}{k}{i} = -B_k^{-1}\ux{g}{k}{i}$, and line search is performed to obtain a step size $\ux{\lambda}{k}{i}$. We take a step $\ux{s}{k}{i} = \ux{\lambda}{k}{i}\ux{d}{k}{i}$. The angle between $\ux{s}{k}{i}$ and $-\ux{g}{k}{i}$ is denoted $\ux{\theta}{k}{i}$. As $B_k$ is positive definite, $\ux{\theta}{k}{i} \in \lb 0, \frac{\pi}{2})$.

After taking $q$ steps, the matrix $B_k$ is updated. Let $G_k = G(\ux{x}{k}{q+1})$ denote the Hessian at the final iterate, and form the matrix $S_k = \lb \ux{s}{k}{1} \ldots \ux{s}{k}{q} \rb$. We apply the \textsc{filtersteps} procedure (Algorithm~\ref{alg:FILTER}) to $S_k$, which returns a subset $D_k$ of the columns of $S_k$ satisfying $\sigma_i^2 \geq \tau \|s_i\|^2$, where $s_i$ is the $i$-th column of $D_k$ and $\sigma_i^2$ is the $i$-th diagonal entry of the $L\Sigma L^T$ decomposition of $D_k^TG_kD_k$. $\tau > 0$ is a parameter which controls the strictness of the filtering; a small value of $\tau$ permits $D_k$ to contain steps that are closer to being linearly dependent, as well as steps with smaller curvature. In essence, \textsc{filtersteps} iteratively computes the $L \Sigma L^T$ decomposition of $S_k^TG_kS_k$ and discards columns of $S_k$ corresponding to small diagonal entries, with the remaining columns forming $D_k$.

Define $q_k$ to be the number of columns of $D_k$. If $D_k$ is the empty matrix (all columns were removed), then no update is performed and $B_{k+1} = B_k$. If $D_k$ is not empty, the matrix $B_k$ is updated to have the same action as the Hessian $G_k$ on the column space of $D_k$, or equivalently,
\begin{equation}\label{eq:SKETCH}
	B_{k+1}D_k = G_kD_k.
\end{equation}
Let $D = D_k, G = G_k$. The formula for the update is given by
\begin{equation}\label{eq:BFGS_UPD}
	B_{k+1} = B_k - B_kD(D^TB_kD)^{-1}D^TB_k + G D(D^TGD)^{-1}D^TG.
\end{equation}
This formula is invariant under a change of basis of $\opn{Col}(D_k)$, so we can choose $D_k$ to be any matrix with the same column space. To see this, observe that a change of basis is given by $D_k P$ for an invertible $q \times q$ matrix $P$. The update (\ref{eq:BFGS_UPD}) for the matrix $D_kP$ is given by
\begin{align*}
	B_{k+1} &= B_k - B_kDP(P^TD^TB_kDP)^{-1}P^TD^TB_k + G DP(P^TD^TGDP)^{-1}P^TD^TG \\
	&= B_k - B_kD(D^TB_kD)^{-1}D^TB_k + G D(D^TGD)^{-1}D^TG.
\end{align*}

On the other hand, the matrix $D_k$ obtained from filtering $S_k$ is \emph{not} invariant under a change of basis of $S_k$, and it is possible to control the number of columns removed by selecting an appropriate basis for $S_k$. We chose to take $S_k = \lb \ux{s}{k}{1} \ldots \ux{s}{k}{q}\rb$ in order to retain control over the ratio $\opn{det}(D_k^TG_kD_k)/\opn{det}(D_k^TB_kD_k)$, which is crucial for our theoretical analysis. We also note that in \cite{GGR}, two other choices for the columns of $D_k$ were studied for use in the Stochastic Block L-BFGS method, and the results reported there showed that the choice $D_k = \lb \ux{s}{k}{1} \ldots \ux{s}{k}{q} \rb$ worked best.

As is the case for standard quasi-Newton updates, there are many possible updates that satisfy equation (\ref{eq:SKETCH}). The specific Block BFGS update (\ref{eq:BFGS_UPD}) is derived as follows. Let $H_k = B_k^{-1}$ be the approximation of the inverse Hessian. In contrast with the classical BFGS update, the update (\ref{eq:BFGS_UPD}) is chosen so that $H_{k+1}$ is the nearest matrix to $H_k$ in a weighted norm, satisfying the system of sketching equations $H_{k+1}G_kD_k = D_k$ rather than a set of secant equations. That is, $H_{k+1}$ is the solution to the minimization problem
\begin{equation}\label{eq:VM}
	\begin{array}{rl} \min\limits_{\wt{H} \in \RR^{n \times n}} &\|\wt{H} - H_k\|_{G_k} \\  \text{s.t } &\wt{H} = \wt{H}^T, \wt{H}G_kD_k = D_k\end{array}
\end{equation}
where $\|\cdot\|_{G_k}$ is the norm $\|X\|_{G_k} = \opn{Tr}(XG_kX^TG_k)$, in analogy with the classical BFGS update. This norm is induced by an inner product, so $H_{k+1}$ is an orthogonal projection onto the subspace $\{\wt{H} \in \Sym^n: \wt{H}G_kD_k = D_k\}$. In Appendix~\ref{appdx:UPD_FORM}, it is shown that $H_{k+1}$ has the explicit formula
\begin{equation}\label{eq:BFGS_INV_UPD}
	H_{k+1} = D(D^TGD)^{-1}D^T + (I - D(D^TGD)^{-1}D^TG)H_k(I - GD(D^TGD)^{-1}D^T).
\end{equation}
Taking the inverse yields formula (\ref{eq:BFGS_UPD}). Moreover, as shown in \cite{SCH}, we have
\begin{LMA}\label{BPosDef}
	If $B_k$ ($H_k$) and $D_k^TG_kD_k$ are positive definite, then the Block BFGS update (\ref{eq:BFGS_UPD}) for $B_{k+1}$ ((\ref{eq:BFGS_INV_UPD}) for $H_{k+1}$) is positive definite.
\end{LMA}
\begin{proof}
	Our proof is adapted from Theorem 3.1 of \cite{SCH}. Let $z \in \RR^n$, and define $w = D_k^Tz, v = z - G_kD_k(D_k^TG_kD_k)^{-1}w$. Using formula (\ref{eq:BFGS_INV_UPD}), we find that
	$$z^TH_{k+1}z = w^T(D_k^TG_kD_k)^{-1}w + v^T H_k v$$
	so $z^TH_{k+1}z \geq 0$. Furthermore, $z^TH_{k+1}z = 0$ only if both $w = 0$ and $v = 0$, in which case $z = 0$. Hence $H_{k+1}$ is positive definite.
\end{proof}

In Section~\ref{sec:CONVERGENT}, we show that Block BFGS converges even if $B_k = B_{k+1} = \ldots$ is stationary. In Section~\ref{sec:SUPERLINEAR}, we show that when $f$ is strongly convex, the parameter $\tau$ can be naturally chosen so an update is always performed, and the convergence is superlinear.

In theory, \textsc{filtersteps} is required to ensure that the update (\ref{eq:BFGS_UPD}) exists. However, in practice, one is unlikely to encounter linearly dependent directions, or directions lying exactly in the null space of $G_k$. Thus, one may omit \textsc{filtersteps} unless there is reason to believe that $G_k$ is singular and problems will arise. However, filtering may improve numerical stability, by removing nearly linearly dependent steps from $D_k$.
\subsection{Rolling Block BFGS}\label{sub:RBBFGS}

Block BFGS uses the same matrix $B_k$ throughout each block of $q$ steps. We could also add information from these steps immediately, at the cost of doing far more updates. This variant, \emph{Rolling Block BFGS}, performs a block update after every step, using a subset $D_k$ of the previous $q$ steps. $D_k$ is formed by adding $s_k$ as the first column of $D_{k-1}$, removing $s_{k-q}$ if present, and filtering.

In general, one might consider schemes for interleaving standard BFGS updates with periodic block updates, to capture additional second-order information.

\section{Convergence of Block BFGS}\label{sec:CONVERGENT}
In this section we prove that Block BFGS with inexact Armijo-Wolfe line searches converges under the same conditions as does the classical BFGS method. These conditions are given in Assumption 1.
\paragraph{\textbf{Assumption 1}}
\begin{enumerate}
	\item $f$ is convex, twice differentiable, and bounded below.
	\item For all $x$ in the level set $\Omega = \{x \in \RR^n: f(x) \leq f(x_1)\}$, the Hessian satisfies $G(x) \preceq MI$, or equivalently, $g(x)$ is Lipschitz continuous with Lipschitz constant $M$.
\end{enumerate}

The main goal of this section is to prove the following theorem. The concept of our proof is similar to the analysis given by Powell \cite{POW} for the classical BFGS method.
\begin{THM}\label{CONVERGENT}
	Let $f$ be a function satisfying Assumption 1, and let $\{x_k\}_{k=1}^\infty$ denote the sequence of all iterates produced by Block BFGS. Then $\liminf_k \|g_k\| = 0$.
\end{THM}

We begin by proving several lemmas. The first two are well known; see \cite{BNY, POW}.

\begin{LMA}\label{SUMMABLE}
	$\sum_{k=1}^\infty \la -g_k, s_k \ra < \infty$, and therefore $\la -g_k, s_k \ra \rightarrow 0$.
\end{LMA}
\begin{proof}
	From the Armijo condition (\ref{armijo}), $\la-g_k, s_k\ra = \lambda_k\la -g_k, d_k\ra \leq (1/\alpha) (f_k - f_{k+1})$. As $f$ is bounded below,
	$$\sum_{k=1}^\infty \la -g_k, s_k\ra \leq (1/\alpha) \sum_{k=1}^\infty (f_k - f_{k+1}) \leq (1/\alpha)(f_1 - \lim_{k \rightarrow \infty} f_k) < \infty.$$
\end{proof}

\begin{LMA}\label{LIPSCH_KEY}
	If the gradient $g(x)$ is Lipschitz continuous with constant $M$, then for $c_1 = \frac{1 - \beta}{M}$, we have $\|s_k\| \geq c_1 \|g_k\| \cos \theta_k$.
\end{LMA}
\begin{proof}
	Let $y_k = g_{k+1} - g_k$. From the Wolfe condition (\ref{wolfe}),
	$$\la y_k, s_k \ra = \la g_{k+1},s_k \ra - \la g_k, s_k \ra \geq (1-\beta)\la -g_k, s_k\ra.$$
	By the Lipschitz continuity of the gradient, $\|y_k\| \leq M\|s_k\|$. Therefore
	$$(1-\beta) \|g_k\|\|s_k\| \cos \theta_k = (1-\beta)\la -g_k, s_k \ra \leq \la y_k, s_k \ra \leq M\|s_k\|^2$$
	yielding $\|s_k\| \geq c_1 \|g_k\| \cos \theta_k$.
\end{proof}

It is possible that $D_k$ is empty for all $k \geq k_0$, and no further updates are made to $B_{k_0}$. This may occur, for example, if $G(x)$ has arbitrarily small eigenvalues and $\tau$ is chosen to be large. In this case, Block BFGS is equivalent to a scaled gradient method $x_{k+1} = x_k - \lambda_k B_{k_0}^{-1}g_k$ with $B_{k_0}$ a constant positive-definite matrix, for all $k \geq k_0$, which is well-known to converge to a stationary point.

For the remainder of this section, we assume that there is an infinite sequence of updates. In fact, we may further assume that an update is made for every $k$, as one can verify that the propositions of this section continue to hold when we restrict our arguments to the subsequence of $\{B_k\}$ for which updates are made. This simplifies the notation. Note, however, that the same cannot simply be assumed in Section~\ref{sec:SUPERLINEAR}. The results in that section do \emph{not} hold if updates are skipped. However, in Section~\ref{sec:SUPERLINEAR} we are able to choose $\tau$ so as to guarantee that an update is made for every $k$.

\begin{LMA}\label{TRACE}
	Let $c_3 = \opn{Tr}(B_1) + qM$. Then for all $k$,
	$$\opn{Tr}(B_k) \leq c_3k \hspace{1em} \text{and}\hspace{1em} \sum_{j=1}^{k} \opn{Tr}(D_j^TB_j^2D_j(D_j^TB_jD_j)^{-1}) \leq c_3k$$
\end{LMA}
\begin{proof}
	Clearly $\opn{Tr}(B_1) \leq c_3$. Define $E_j = G_j^\half D_j$, and let $P_j = E_j(E_j^TE_j)^{-1}E_j^T$ be the orthogonal projection onto $\opn{Col}(E_j)$, so that $G_jD_j(D_j^TG_jD_j)^{-1}D_j^TG_j = G_j^\half P_j G_j^\half$. For $k \geq 1$, we expand $\opn{Tr}(B_{k+1})$ using Equation~(\ref{eq:BFGS_UPD}):
	\begin{align*}
		0 < \opn{Tr}(B_{k+1}) &= \opn{Tr}(B_1) + \sum_{j=1}^{k} \opn{Tr}(G_j^\half P_j G_j^\half) - \sum_{j=1}^{k} \opn{Tr}(D_j^TB_j^2D_j(D_j^TB_jD_j)^{-1}) \\
		&\leq \opn{Tr}(B_1) + k(qM) - \sum_{j=1}^{k} \opn{Tr}(D_j^TB_j^2D_j(D_j^TB_jD_j)^{-1})
	\end{align*}
	where the first inequality follows from the positive definiteness of $B_{k+1}$ (Lemma~\ref{BPosDef}) and the second inequality follows since $\opn{rank}(P_j)\leq q$, and $\|G_j^\half P_j G_j^\half\| \leq \|G_j\| \|P_j\| \leq M$. This shows $\opn{Tr}(B_{k+1}) \leq c_3(k+1)$ and $\sum_{j=1}^{k} \opn{Tr}(D_j^TB_j^2D_j(D_j^TB_jD_j)^{-1}) \leq c_3k$.
\end{proof}

\begin{LMA}\label{STEPTRACE}
	Let $\ux{s}{k}{i}$ be a step included in $D_k$. Then
	$$\frac{\ux{\lambda}{k}{i} \|\ux{g}{k}{i}\|^2}{ \la -\ux{g}{k}{i}, \ux{s}{k}{i} \ra} \leq \opn{Tr}(D_k^TB_k^2D_k(D_k^TB_kD_k)^{-1})$$
\end{LMA}
\begin{proof}
	By the Gram-Schmidt process applied to the columns of $D_k$, we can find a set of $B_k$-conjugate vectors $\{v_1,\ldots,v_{q_k}\}$ spanning $\opn{Col}(D_k)$ with $v_1 = \ux{s}{k}{i}$. Using the matrix $\lb v_1 \ldots v_{q_k}\rb$ for $D_k$, we have
	\begin{align*}
		D_k^TB_kD_k &= \opn{Diag}( \la \ux{s}{k}{i}, -\ux{\lambda}{k}{i}\ux{g}{k}{i} \ra, \la v_2, B_kv_2\ra, \ldots, \la v_{q_k}, B_kv_{q_k} \ra)
	\end{align*}
	and therefore
	\begin{align*}
		\opn{Tr}(D_k^TB_k^2D_k(D_k^TB_kD_k)^{-1}) &= \sum_{\ell=1}^{q_k} \lb D_k^TB_k^2D_k \rb_{\ell \ell} \lb D_k^TB_kD_k \rb_{\ell \ell}^{-1} \\
		&= \frac{ (\ux{\lambda}{k}{i} \|\ux{g}{k}{i}\|)^2}{\ux{\lambda}{k}{i} \la -\ux{g}{k}{i}, \ux{s}{k}{i} \ra} + \sum_{\ell=2}^{q_k} \frac{\|B_kv_\ell\|^2}{\la v_\ell, B_kv_\ell \ra} \geq \frac{\ux{\lambda}{k}{i} \|\ux{g}{k}{i}\|^2}{ \la -\ux{g}{k}{i}, \ux{s}{k}{i} \ra}
	\end{align*}
\end{proof}
We may assume without loss of generality that $D_k = \lb \ux{s}{k}{1} \ldots \ux{s}{k}{q_k}\rb$.
\begin{CORO}\label{TRPROD}
	$$\prod_{j=1}^{k} \prod_{i=1}^{q_j} \frac{\ux{\lambda}{j}{i} \|\ux{g}{j}{i}\|^2}{ \la -\ux{g}{j}{i}, \ux{s}{j}{i} \ra} \leq (qc_3)^{qk}$$
\end{CORO}
\begin{proof}
	Let $\wh{q}_k = \sum_{j=1}^{k} q_j$, and note that $k \leq \wh{q}_k \leq qk$. Hence, from Lemmas~\ref{TRACE} and \ref{STEPTRACE},
	$$\frac{1}{\wh{q}_k} \sum_{j=1}^{k} \sum_{i=1}^{q_j} \frac{\ux{\lambda}{j}{i} \|\ux{g}{j}{i}\|^2}{ \la -\ux{g}{j}{i}, \ux{s}{j}{i} \ra} \leq \frac{qk}{\wh{q}_k} c_3 \leq qc_3$$
	Applying the arithmetic mean-geometric mean (AM-GM) inequality,
	$$ \left(\prod_{j=1}^{k} \prod_{i=1}^{q_j} \frac{\ux{\lambda}{j}{i} \|\ux{g}{j}{i}\|^2}{ \la -\ux{g}{j}{i}, \ux{s}{j}{i} \ra} \right) \leq (qc_3)^{\wh{q}_k} \leq (qc_3)^{qk}.$$
\end{proof}

\begin{LMA}\label{DETBD}
	$\opn{det}(B_k) \leq \left( \frac{c_3k}{n} \right)^n$ for all $k$.
\end{LMA}
\begin{proof}
	By Lemma~\ref{TRACE}, $\opn{Tr}(B_k) \leq c_3k$. Recall that the trace is equal to the sum of the eigenvalues, and the determinant to the product. Applying the AM-GM inequality to the eigenvalues of $B_k$, we obtain $\opn{det}(B_k) \leq \left( \frac{c_3k}{n} \right)^n$.
\end{proof}

We will need the following two classical results from matrix theory; see \cite{HJ}.

\noindent \textbf{Sylvester's Determinant Identity} Let $A \in \RR^{n \times m}, B \in \RR^{m \times n}$. Then $$\opn{det}(I_n + AB) = \opn{det}(I_m + BA)$$

\noindent \textbf{Sherman-Morrison-Woodbury Formula} Let $A \in \RR^{n \times n}$ and $C \in \RR^{k \times k}$ be invertible, and $U \in \RR^{n \times k}, V \in \RR^{k \times n}$. If $A + UCV$ and $C^{-1} + VA^{-1}U$ are invertible, then $(A + UCV)^{-1} = A^{-1} - A^{-1}U(C^{-1} + VA^{-1}U)^{-1}VA^{-1}$.

\begin{LMA}\label{DETUPDATE}
	$$\opn{det}(B_{k+1}) = \frac{ \opn{det}(D_k^TG_kD_k)}{\opn{det}(D_k^TB_kD_k)} \opn{det}(B_k) $$
\end{LMA}
\begin{proof}
	Let $B = B_k, B^+ = B_{k+1}, D = D_k, G = G_k$. Then
	$$\opn{det}(B^+) = \opn{det}(B) \opn{det}(I + B^{-\half} GD(D^TGD)^{-1}D^TGB^{-\half} - B^\half D(D^TBD)^{-1}D^TB^\half).$$
	Define $X = B^{-\half} GD(D^TGD)^{-1}D^TGB^{-\half}$ and $Y = D^TGD + D^TGB^{-1}GD$. Note that $I + X$ is invertible since $X \succeq 0$ and $I \succ 0$, and $Y$ is invertible since $D^TGD \succ 0$. Thus, we can write
	$$\opn{det}(B^+) = \opn{det}(B) \opn{det}(I + X)\opn{det}(I - (I + X)^{-1}B^\half D(D^TBD)^{-1}D^TB^\half).$$
	Applying Sylvester's determinant identity to each term,
	\begin{align*}
		&\opn{det}(I + X) = \opn{det}(I + (D^TGB^{-\half})(B^{-\half}GD(D^TGD)^{-1})) = \opn{det}(Y)\opn{det}(D^TGD)^{-1} \\
		&\opn{det}(I - (I + X)^{-1}B^\half D(D^TBD)^{-1}D^TB^\half) = \opn{det}(I - D^TB^\half (I + X)^{-1}B^\half D(D^TBD)^{-1})
	\end{align*}
	Applying the Sherman-Morrison-Woodbury formula to $I + X$ with $U = B^{-\half} GD, C = (D^TGD)^{-1}, V = D^TGB^{-\half}$, we obtain $(I + X)^{-1} = I - B^{-\half}GDY^{-1}D^TGB^{-\half}$, so
	\begin{align*}
		\opn{det}(I - (I + X)^{-1}B^\half D(D^TBD)^{-1}D^TB^\half) = \opn{det}(D^TGD)^2 \opn{det}(Y)^{-1} \opn{det}(D^TBD)^{-1}.
	\end{align*}
	Thus $\opn{det}(B^+) = \opn{det}(B) \opn{det}(D^TGD)\opn{det}(D^TBD)^{-1}$ as desired.
\end{proof}

\begin{LMA}\label{DET}
	$$\opn{det}(B_{k+1}) \geq  \left(\prod_{i=1}^{q_k} \frac{1}{\lambda_i} \right) (\tau c_1)^{q_k} \opn{det}(B_k)$$
\end{LMA}
\begin{proof}
	Recall that the columns of $D_k$ satisfy $\sigma_i^2 \geq \tau \|\ux{s}{k}{i}\|^2$, where $\sigma_i$ is the $i$-th diagonal element of the $L\Sigma L^T$ decomposition of $D_k^TG_kD_k$. We have $\opn{det}(D_k^TG_kD_k) = \prod_{i=1}^{q_k} \sigma_i^2$ and $\opn{det}(D_k^TB_kD_k) \leq \prod_{i=1}^{q_k} \lb D_k^TB_kD_k\rb_{ii} = \prod_{i=1}^{q_k} \la \ux{s}{k}{i}, -\ux{\lambda}{k}{i} \ux{g}{k}{i} \ra$. By Lemma~\ref{DETUPDATE},
	\begin{align*}
		\opn{det}(B_{k+1}) &= \opn{det}(B_k) \frac{\opn{det}(D_k^TG_kD_k)}{\opn{det}(D_k^TB_kD_k)} \\
		&\geq \opn{det}(B_k) \frac{ \prod_{i=1}^{q_k} \tau \|\ux{s}{k}{i}\|^2}{\prod_{i=1}^{q_k} \la \ux{s}{k}{i}, -\ux{\lambda}{k}{i} \ux{g}{k}{i} \ra} \geq  \opn{det}(B_k)\prod_{i=1}^{q_k} \frac{\tau}{\ux{\lambda}{k}{i}} \frac{\|\ux{s}{k}{i}\|}{\|\ux{g}{k}{i}\| \cos \ux{\theta}{k}{i} }.
	\end{align*}
	By Lemma \ref{LIPSCH_KEY}, $\frac{\|\ux{s}{k}{i}\|}{\|\ux{g}{k}{i}\| \cos \ux{\theta}{k}{i}} \geq c_1$. Hence $\opn{det}(B_{k+1})  \geq  \left(  \prod\limits_{i=1}^{q_k} \frac{1}{\ux{\lambda}{k}{i}} \right)(\tau c_1)^{q_k} \opn{det}(B_k)$.
\end{proof}

\begin{CORO}\label{DETPROD}
	$$\opn{det}(B_{k+1}) \geq (\tau c_1)^{qk} \opn{det}(B_1) \prod_{j=1}^{k} \prod_{i=1}^{q_j} \frac{1}{\ux{\lambda}{j}{i}} $$
\end{CORO}

\begin{CORO}\label{LASTBOUND}
	There exists a constant $c_4$ such that for all $k$,
	$$\prod_{j=1}^{k} \prod_{i=1}^{q_j} \frac{\|\ux{g}{j}{i}\|^2}{ \la -\ux{g}{j}{i}, \ux{s}{j}{i} \ra} \leq c_4^k$$
\end{CORO}
\begin{proof}
	Multiplying the inequalities of Corollary \ref{TRPROD} and Lemma \ref{DETBD}, we obtain
	$$\left( \prod_{j=1}^{k} \prod_{i=1}^{q_j} \frac{\ux{\lambda}{j}{i} \|\ux{g}{j}{i}\|^2}{ \la -\ux{g}{j}{i}, \ux{s}{j}{i} \ra} \right)  \left( \frac{\opn{det}(B_{k+1})}{\opn{det}(B_1)} \right) \leq (qc_3)^{qk}  \left( \frac{ (c_3(k+1)/n)^n}{\opn{det}(B_1)} \right) \leq \rho_1^k$$
	for some constant $\rho_1$. Using the lower bound of Corollary~\ref{DETPROD}, we also obtain
	\begin{align*}
		\left( \prod_{j=1}^{k} \prod_{i=1}^{q_j} \frac{\ux{\lambda}{j}{i} \|\ux{g}{j}{i}\|^2}{ \la -\ux{g}{j}{i}, \ux{s}{j}{i} \ra} \right)
		\left( \frac{\opn{det}(B_{k+1})}{\opn{det}(B_1)} \right)
		&\geq
		\left( \prod_{j=1}^{k} \prod_{i=1}^{q_j} \frac{\ux{\lambda}{j}{i} \|\ux{g}{j}{i}\|^2}{ \la -\ux{g}{j}{i}, \ux{s}{j}{i} \ra} \right)
		\cdot
		(\tau c_1)^{qk} \prod_{j=1}^{k} \prod_{i=1}^{q_j} \frac{1}{\ux{\lambda}{j}{i}} \\
		&= (\tau c_1)^{qk} \left( \prod_{j=1}^{k} \prod_{i=1}^{q_j} \frac{\|\ux{g}{j}{i}\|^2}{ \la -\ux{g}{j}{i}, \ux{s}{j}{i} \ra} \right)
	\end{align*}
	Take $c_4 = \frac{\rho_1}{(\tau c_1)^q}$, whence $\prod_{j=1}^{k} \prod_{i=1}^{q_j} \frac{\|\ux{g}{j}{i}\|^2}{ \la -\ux{g}{j}{i}, \ux{s}{j}{i} \ra} \leq c_4^k$.
\end{proof}

Finally, we can establish our main result.

\begin{proof}(of Theorem~\ref{CONVERGENT})
	Assume to the contrary that $\|\ux{g}{k}{i}\|$ is bounded away from zero. Lemma \ref{SUMMABLE} implies that $\la \ux{g}{k}{i}, -\ux{s}{k}{i} \ra \rightarrow 0$. Thus, there exists $k_0$ such that for $k \geq k_0$, $\frac{ \|\ux{g}{k}{i}\|^2}{\la \ux{g}{k}{i}, -\ux{s}{k}{i} \ra} > c_4 + 1$. This contradicts Corollary~\ref{LASTBOUND}, as $\prod_{j=1}^{k} \prod_{i=1}^{q_j} \frac{\|\ux{g}{j}{i}\|^2}{ \la -\ux{g}{j}{i}, \ux{s}{j}{i} \ra} \leq c_4^k$ for all $k$. We conclude that $\liminf_k \|g_k\| = 0$.
\end{proof}

A similar analysis shows that Rolling Block BFGS (Section~\ref{sub:RBBFGS}) converges.
\begin{THM}
	Assume $f$ satisfies Assumption 1. Then the sequence $\{g_k\}_{k=1}^\infty$ produced by Rolling Block BFGS satisfies $\liminf_k \|g_k\| = 0$.
\end{THM}
\begin{proof}
	By the calculations for Corollary~\ref{TRPROD}, we have $\prod_{j=1}^{k} \frac{\lambda_j \|g_j\|^2}{\la -g_j, s_j\ra} \leq c_3^k$.
	
	$D_k$ is produced by adding column $s_k$ to $D_{k-1}$, removing $s_{k-q}$ if present, and then running Algorithm~\ref{alg:FILTER}. Without loss of generality, assume that $D_k = \lb s_k \ldots s_{k - q_k+1} \rb$. By definition, $B_k$ satisfies $B_kD_{k-1} = G_{k-1}D_{k-1}$. Thus, we have
	$$\opn{det}(D_k^TB_kD_k) \leq \prod_{i=0}^{q_k - 1} \la s_{k-i}, B_ks_{k-i} \ra = \la s_k, B_k s_k \ra \prod_{i=1}^{q_k-1} \la s_{k-i}, G_{k-1}s_{k-i}\ra$$
	which gives an analogue of Lemma~\ref{DET}:
	\begin{align*}
		\opn{det}(B_{k+1}) \geq \frac{ \prod_{i=0}^{q_k - 1} \tau \|s_{k - i}\|^2}{ \la s_k, -\lambda_k g_k \ra \prod_{i=1}^{q_k-1} \la s_{k-i}, G_{k-1}s_{k-i} \ra }  \opn{det}(B_k) \geq \frac{1}{\lambda_k}\frac{c_1 \tau^q}{M^{q-1}} \opn{det}(B_k).
	\end{align*}
	Thus $\opn{det}(B_{k+1}) \geq \left( \frac{c_1\tau^q}{M^{q-1}} \right)^k \opn{det}(B_1) \prod_{j=1}^{k} \frac{1}{\lambda_k}$. The remainder of the proof follows exactly as in the proofs of  Corollary~\ref{LASTBOUND} and Theorem~\ref{CONVERGENT}.
\end{proof}

\section{Superlinear Convergence of Block BFGS}\label{sec:SUPERLINEAR}
In this section we show that Block BFGS converges $Q$-superlinearly under the same conditions as does BFGS, namely, that $f$ is strongly convex in a neighborhood of its minimizer, and its Hessian is Lipschitz continuous. We use the characterization of superlinear convergence given by Dennis and Mor\'{e} \cite{DM1}, and employ an argument similar to the analysis used by Griewank and Toint \cite{GT} for partitioned quasi-Newton updates. The following conditions, which strengthen Assumption 1, will apply to $f$ throughout this section.
\paragraph{\textbf{Assumption 2}}
\begin{enumerate}
	\item $f$ is convex and twice differentiable, with $G(x) \preceq MI$ on the level set $\{x \in \RR^n: f(x) \leq f(x_1)\}$.
	\item $f$ has a minimizer $x_\ast$ for which $G(x_\ast)$ is non-singular. Note that this implies $x_\ast$ is unique.
	\item $G(x)$ is Lipschitz in a neighborhood of $x_\ast$, with Lipschitz constant $\mu$.
\end{enumerate}

Since Assumption 2 is stronger than Assumption 1, \Cref{CONVERGENT} implies that the iterates produced by Block BFGS converge to the unique stationary point $x_\ast$. The continuity of $G(x)$ and the fact that $G(x_\ast)$ is non-singular imply that $f$ is strongly convex in a neighborhood $S$ of $x_\ast$. Superlinear convergence is an asymptotic property, so we may restrict our attention to the tail of the sequence $\{x_k\}$, contained in $S$. That is, we may assume without loss of generality that all iterates $\{x_k\}$ lie in a region $S$ on which $f$ is strongly convex, with
$$mI \preceq G(x) \preceq MI \hspace{2em} \forall x \in S$$
for constants $0 < m \leq M$.

In this section, we assume $\tau \leq m$, where $\tau$ is the parameter in \textsc{filtersteps}. Since $\sigma_1^2 = \lb S_k^TG_k S_k \rb_{11} = \la \ux{s}{k}{1}, G_k\ux{s}{k}{1} \ra \geq m \|\ux{s}{k}{1}\|^2$, the first column of $D_k$ is never removed by \textsc{filtersteps}. This guarantees that an update is always performed. The choice of $\tau$ is important and can impact superlinear convergence; we give a detailed discussion in the remarks concluding this section.

\begin{THM}\label{SUPERLINEAR}
	Let $f$ be a function satisfying Assumption 2. Block BFGS converges $Q$-superlinearly along the subsequence of steps in $D_k$; that is,
	$$\lim_{\substack{k \rightarrow \infty \\ i \in D_k}} \frac{ \|\ux{x}{k}{i+1} - x_\ast\|}{\|\ux{x}{k}{i} - x_\ast\|} = 0.$$
\end{THM}

To clarify the statement of this theorem, the quotients $\| \ux{x}{k}{i+1} - x_\ast\| / \| \ux{x}{k}{i} - x_\ast\|$ in the subsequence are those for which $\ux{s}{k}{i}$ is in $D_k$. If every step is included in $D_k$, then we have $Q$-superlinear convergence for the sequence of points $\{\ux{x}{k}{i}\}$ in the usual sense. To give an example of the contrary, suppose the step $\ux{s}{10}{2}$ is removed by filtering; then the quotient $\|\ux{x}{10}{3} - x_\ast\|/\|\ux{x}{10}{2} - x_\ast\|$ is not captured in the subsequence. In theory, one step is guaranteed per block $D_k$, but we note that in practice, $D_k$ contains all or nearly all steps for every $k$.

We begin by showing that Block BFGS converges $R$-linearly. The first three lemmas are well known; see \cite{BNY,POW}. These three lemmas apply to every step, and thus we write $x_{k+1}$ for the iterate immediately following $x_k$, instead of using superscripts.

\begin{LMA}\label{KEY}
	For $c_1 = \frac{1-\beta}{M}$ and $c_2 = \frac{2(1-\alpha)}{m}$, 
	$$c_1 \|g_k\| \cos \theta_k \leq \|s_k\| \leq c_2 \|g_k\| \cos\theta_k$$
\end{LMA}
\begin{proof}
	By Taylor's theorem, there exists a point $\wt{x}$ on the line segment joining $x_k, x_{k+1}$ such that $f(x_{k+1}) = f(x_k) + \la g_k, s_k \ra + \frac{1}{2}s_k^TG(\wt{x})s_k$. From (\ref{armijo}), $f(x_{k+1}) - f(x_k) \leq \alpha \la g_k, s_k \ra$, so $(1-\alpha) \la -g_k, s_k\ra \geq \frac{1}{2} s_k^TG(\wt{x})s_k \geq \frac{1}{2}m\|s_k\|^2$. Rearranging yields $\|s_k\| \leq c_2 \|g_k\| \cos \theta_k$. The lower bound was shown in Lemma~\ref{LIPSCH_KEY}. 
\end{proof}

\begin{LMA}\label{gStrongConvex}
	For any $x \in S$, $\|g(x)\|^2 \geq 2m(f(x) - f_\ast)$.
\end{LMA}
\begin{proof}
	The result is immediate if $x = x_\ast$, so assume $x \neq x_\ast$. By Taylor's theorem, there exists a point $\wt{x}$ on the line segment joining $x, x_\ast$ such that $f(x_\ast) = f(x) + g(x)^T(x_\ast - x) + \frac{1}{2}(x_\ast - x)^TG(\wt{x})(x_\ast - x)$, in which case
	$$g(x)^T(x - x_\ast) = f(x) - f_\ast + \half (x_\ast - x)^TG(\wt{x})(x_\ast - x) \geq f(x) - f_\ast + \half m \|x - x_\ast\|^2.$$
	Using the Cauchy-Schwarz inequality, we find that $\|g(x)\|\|x - x_\ast\| \geq f(x) - f_\ast + \half m \|x - x_\ast\|^2$. Applying the AM-GM inequality and squaring yields $\|g(x)\|^2 \geq 2m(f(x) - f_\ast)$.
\end{proof}

\begin{LMA}\label{REDUCTION}
	$$f_{k+1} - f_\ast \leq (1 - 2\alpha m c_1 \cos^2 \theta_k ) (f_k - f_\ast)$$
\end{LMA}
\begin{proof}
	The Armijo condition~(\ref{armijo}) and Lemma~\ref{KEY} imply that
	$$f_{k+1} - f_k \leq \alpha \la g_k, s_k \ra = -\alpha \|g_k\|\|s_k\|\cos\theta_k \leq -\alpha c_1\|g_k\|^2 \cos^2\theta_k.$$
	By Lemma~\ref{gStrongConvex}, $\|g_k\|^2 \geq 2 m (f_k  - f_\ast)$. Hence $f_{k+1} - f_\ast \leq \left(1 - 2\alpha m c_1 \cos^2 \theta_k \right)(f_k - f_\ast)$.
\end{proof}

Define $r_k = \|\ux{x}{k}{q+1} - x_\ast\|$. $R$-linear convergence implies that the errors $r_k$ diminish to zero rapidly enough that $\sum_{k=1}^\infty r_k < \infty$, a key property.

\begin{THM}\label{RLINEAR}
	There exists $\delta < 1$ such that $f(\ux{x}{k}{q+1}) - f_\ast \leq \delta^k (f(\ux{x}{1}{1}) - f_\ast)$, and thus $\sum_{k=1}^\infty r_k < \infty$.
\end{THM}
\begin{proof}
	From Lemma~\ref{LASTBOUND}, $\prod_{j=1}^{k} \prod_{i=1}^{q_j} \frac{\|\ux{g}{j}{i}\|}{\|\ux{s}{j}{i}\| \cos \ux{\theta}{j}{i} } \leq c_4^k$. Lemma~\ref{KEY} gives the upper bound $\|\ux{s}{j}{i}\| \leq c_2 \| \ux{g}{j}{i}\| \cos \ux{\theta}{j}{i}$. Substituting, we find
	$$\prod_{j=1}^{k} \prod_{i=1}^{q_j} \cos^2 \ux{\theta}{j}{i} \geq \left( \frac{1}{c_2^qc_4} \right)^k.$$
	From this, we see that at least $\frac{1}{2}k$ of the angles must satisfy $\cos^2 \ux{\theta}{j}{i} \geq \left( \frac{1}{c_2^qc_4} \right)^2$.
	
	By Lemma~\ref{REDUCTION}, $f(\ux{x}{k}{i+1}) - f_\ast \leq (1 - 2\alpha m c_1 \cos^2 \theta_k ) (f(\ux{x}{k}{i}) - f_\ast)$. Using our bound on the angles,
	$$f(\ux{x}{k}{q+1}) - f_\ast \leq \left(1 - 2\alpha m c_1 \left( \frac{1}{c_2^qc_4} \right)^2 \right)^{\frac{1}{2}k} (f(\ux{x}{1}{1}) - f_\ast).$$
	Hence, we may take $\delta = \left(1 - \frac{2\alpha m c_1}{c_2^{2q}c_4^2} \right)^{1/2}$. The strong convexity of $f$ implies that $\frac{1}{2}m\|x - x_\ast\|^2 \leq f(x) - f_\ast \leq \frac{1}{2}M\|x - x_\ast\|^2$, so we have $r_k \leq (\sqrt{\delta})^k \sqrt{\frac{M}{m}} \|\ux{x}{1}{1} - x_\ast\|$. Therefore $\sum_{k=1}^\infty r_k < \infty$.
\end{proof}

The classical BFGS method is invariant under a linear change of coordinates. It is easy to verify that Block BFGS also has this invariance, so we may assume without loss of generality that $G(x_\ast) = I$. This greatly simplifies the following calculations. Given that Theorem~\ref{CONVERGENT} implies that Block BFGS converges, we will also assume that the iterates lie in the region around $x_\ast$ where $G(x)$ is Lipschitz continuous.
\begin{LMA}\label{LIPSCHITZ}
	For any $v \in \RR^n$, $\|(G_k - I)v\| \leq \mu r_k\|v\|$.
\end{LMA}
\begin{proof}
	Since $G(x_\ast) = I$,
	$$\|(G_k - I)v\| \leq \| G(\ux{x}{k}{q+1}) - G(x_\ast)\| \|v\| \leq \mu \|\ux{x}{k}{q+1} - x_\ast\| \|v\| = \mu r_k \|v\|.$$
\end{proof}

The following notion is useful in our analysis. Define $\wt{B}_{k+1}$ to be the matrix obtained by performing a Block BFGS update on $B_k$ with $G_k = G(x_\ast)$. Since we assumed $G(x_\ast) = I$, we have the explicit formula
$$\wt{B}_{k+1} = B_k - B_kD_k(D_k^TB_kD_k)^{-1}D_k^TB_k + D_k(D_k^TD_k)^{-1}D_k^T$$
and its inverse $\wt{H}_{k+1}$ is given by
$$\wt{H}_{k+1} = D_k(D_k^TD_k)^{-1}D_k^T + (I - D_k(D_k^TD_k)^{-1}D_k^T)H_k(I - D_k(D_k^TD_k)^{-1}D_k^T).$$

\begin{LMA}\label{B_IMPR}
	Let $B = B_k, \wt{B} = \wt{B}_{k+1}, D = D_k$. Define the following orthogonal projections:
	\begin{enumerate}
		\item $P = B^\frac{1}{2}D( D^TBD)^{-1}D^TB^\frac{1}{2}$, the projection onto $\opn{Col}(B^\half D)$.
		\item $P_D = D(D^TD)^{-1}D^T$, the projection onto $\opn{Col}(D)$.
		\item $P_B = BD(D^TB^2D)^{-1}D^TB$, the projection onto $\opn{Col}(BD)$.
	\end{enumerate}
	Then
	$$\|B - I\|_F^2 - \|\wt{B} - I\|_F^2 = \|P_B - B^\half P B^\half\|_F^2 + 2 \opn{Tr}( B(B^\half P B^\half) - (B^\half P B^\half)^2)$$
	Furthermore, $\opn{Tr}( B(B^\half P B^\half) - (B^\half P B^\half)^2) \geq 0$, and thus $\|\wt{B} - I\|_F \leq \|B - I\|_F$.
\end{LMA}
\begin{proof}
	Expand the Frobenius norm and use the identity $\opn{Tr}(BP_D) = \opn{Tr}(B^\half P B^\half P_D)$ to obtain
	\begin{align*}
		\|B - I\|_F^2 - \|\wt{B} - I\|_F^2 &= 2\opn{Tr}(B(B^\half P B^\half)) - \opn{Tr}((B^\half P B^\half)^2) - 2\opn{Tr}(B^\half P B^\half) \\
		&\hspace{1em} - \opn{Tr}(P_D^2) + 2\opn{Tr}(P_D) \\
		&=  2\opn{Tr}(B(B^\half P B^\half)) - 2\opn{Tr}((B^\half P B^\half)^2) \\
		&\hspace{1em} + \opn{Tr}((B^\half P B^\half)^2) - 2\opn{Tr}(B^\half P B^\half) + \opn{Tr}(I) \\
		&\hspace{1em} - \opn{Tr}(P_D^2) + 2\opn{Tr}(P_D) - \opn{Tr}(I)
	\end{align*}
	Factoring the above equation produces
	$$\|B - I\|_F^2 - \|\wt{B} - I\|_F^2 = \|I - B^\half P B^\half \|_F^2 - \|I - P_D\|_F^2 + 2\opn{Tr}(B (B^\half P B^\half) - (B^\half P B^\half)^2).$$
	Let $P_B^\perp$ be the projection onto the orthogonal complement of $\opn{Col}(BD)$; hence $I = P_B + P_B^\perp$. Since $\la P_B^\perp, B^\half P B^\half \ra = \opn{Tr}(P_B^\perp BD(D^TBD)^{-1}D^TB) = 0$, we have $\|I - B^\half P B^\half \|_F^2 = \|P_B - B^\half P B^\half\|_F^2 + \|P_B^\perp\|_F^2$. The Frobenius norm of an orthogonal projection is equal to the square root of its rank, and thus
	$$\|I - B^\half P B^\half\|_F^2 - \|I - P_D\|_F^2 = \|P_B - B^\half P B^\half \|_F^2 + \|P_B^\perp\|_F^2 - \|I - P_D\|_F^2 = \|P_B - B^\half P B^\half\|_F^2$$
	This gives the desired equation. Now, observe that
	\begin{align*}
		\opn{Tr}( B(B^\half P B^\half) - (B^\half P B^\half)^2) &= \opn{Tr}(BPB(I - P)) \\
		&= \opn{Tr}((I - P)BPB(I - P)) \geq 0
	\end{align*}
	where in the second equality we have used that $I - P$ is the orthogonal projection onto $\opn{Col}(B^\half D)^\perp$, and is therefore idempotent. This proves $\|\wt{B} - I\|_F \leq \|B - I\|_F$.
\end{proof}

Intuitively, $\wt{B}_{k+1}$ and $\wt{H}_{k+1}$ should be closer approximations of $I$ than $B_k$ and $H_k$. This is made precise in the next lemma.
\begin{LMA}\label{IMPR}
	$\|\wt{B}_{k+1} - I\|_F \leq \|B_k - I\|_F$ and $\|\wt{H}_{k+1} - I\|_F \leq \|H_k - I\|_F$.
\end{LMA}
\begin{proof}
	That $\|\wt{B}_{k+1} - I\|_F \leq \|B_k - I\|_F$ was shown in Lemma~\ref{B_IMPR}. Clearly $\|\wt{H}_{k+1} - I\|_F \leq \|H_k - I\|_F$, as $\wt{H}_{k+1}$ is defined as the orthogonal projection of $H_k$ onto the subspace of matrices $\{\wt{H} \in \Sym^n: \wt{H}D_k = D_k\}$, which contains $I$ (see (\ref{eq:VM})).
\end{proof}

\begin{LMA}\label{GROWTH}
	There exists an index $k_0$ and constants $\kappa_1, \kappa_2$ such that $\|B_{k+1} - \wt{B}_{k+1}\|_F \leq \kappa_1 r_k$ and $\|H_{k+1} - \wt{H}_{k+1}\|_F \leq ( \|H_k - I\|_F + 1)\kappa_2 r_k$ for all $k \geq k_0$.
\end{LMA}
\begin{proof}
	Define $\Delta_k = (G_k - I)D_k$. For brevity, let $\wt{B} = \wt{B}_{k+1}, \wt{H} = \wt{H}_{k+1}, H = H_k, D = D_k, G = G_k$, and $\Delta = \Delta_k$. We may assume the columns of $D$ are orthonormal, so $D^TD = I$. By Lemma~\ref{LIPSCHITZ}, every column $\delta_i$ of $\Delta$ satisfies $\|\delta_i\| \leq \mu r_k$, which gives the useful bounds $\|\Delta\|, \|\Delta^T\| \leq \mu\sqrt{q}r_k$. This stems from the fact that a matrix $A$ of rank $q$ satisfies $\|A\| = \|A^T\| \leq \|A\|_F \leq \sqrt{q}\|A\|$, which we will use frequently.
	
	To prove the first inequality, we write
	\begin{align*}
		\|B_{k+1} - \wt{B}\|_F &= \|GD(D^TGD)^{-1}D^TG - DD^T\|_F \\
		&= \| GD(I + D^T\Delta)^{-1}D^TG - DD^T\|_F.
	\end{align*}
	By the Sherman-Morrison-Woodbury formula, $(I + D^T\Delta)^{-1} = I - D^T(I + \Delta D^T)^{-1}\Delta$. Let $X = I + \Delta D^T$. Inserting this expression and using the triangle inequality, we have
	\begin{align*}
		\| GD(I + D^T\Delta)^{-1}D^TG - DD^T\|_F &= \|GDD^TG - DD^T - GDD^TX^{-1}\Delta D^T G\|_F \\
		&\leq \|GDD^TG - DD^T\|_F + \|GDD^TX^{-1}\Delta D^TG\|_F
	\end{align*}
	By a routine calculation,
	\begin{align*}
		\|GDD^TG - DD^T\|_F &= \|\Delta \Delta^T + \Delta D^T + D \Delta^T\|_F,
	\end{align*}
	hence $\|GDD^TG - DD^T\|_F \leq \rho_2 r_k$ for some constant $\rho_2$. 
	
	To bound the Frobenius norm of the other term, we bound its operator norm. Since $\Delta_k \rightarrow 0$ as $r_k \rightarrow 0$, there exists an index $k_0$ such that for $k \geq k_0$,
	\begin{enumerate}
		\item $\|X - I\| \leq \half$, so $\|X^{-1}\| \leq 2$, and
		\item $\|G - I\| \leq 1$, so $\|G\| \leq 2$
	\end{enumerate}
	in which case $\|GDD^TX^{-1}\Delta D^TG\| \leq \rho_3 r_k$ for some $\rho_3$. Taking $\kappa_1 = \rho_2 + \sqrt{q}\rho_3$, we then have $\|B_{k+1} - \wt{B}\|_F \leq \kappa_1 r_k$ for all $k \geq k_0$.
	
	A similar analysis applies to $\|H_{k+1} - \wt{H}\|_F$. Using the triangle inequality,
	\begin{align*}
		\|H_{k+1} - \wt{H}\|_F &\leq \| D(D^TGD)^{-1}D^T - DD^T\|_F \\
		&~+ \|(D(D^TGD)^{-1}D^TG - DD^T)H + H(GD(D^TGD)^{-1}D^T - DD^T)\|_F \\
		&~+ \|D(D^TGD)^{-1}D^TGHGD(D^TGD)^{-1}D^T - DD^THDD^T\|_F
	\end{align*}
	We bound each of the three terms. As before, $(D^TGD)^{-1} = I - D^TX^{-1}\Delta$, so we have $\|D(D^TGD)^{-1}D^T - DD^T\|_F = \|DD^TX^{-1}\Delta D^T\|_F$. For $k \geq k_0$, $\|X^{-1}\| \leq 2$, so $\|D(D^TGD)^{-1}D^T - DD^T\|_F \leq \rho_4 r_k$ for some $\rho_4$.
	
	For the second term, observe that
	\begin{align*}
		GD(D^TGD)^{-1}D^T - DD^T &= \Delta D^T - DD^TX^{-1}\Delta D^T - \Delta D X^{-1} \Delta D^T.
	\end{align*}
	Hence, the norm of the second term is bounded above by $\rho_5 r_k \|H\|$ for some $\rho_5$.
	
	Finally, we bound the operator norm of the third term. Factoring out $D$ and $D^T$ on the left and right, we can write the inside term as
	\begin{align*}D^TGHGD - D^THD &- (D^TX^{-1}\Delta D^TGHGD + D^TGHGD D^TX^{-1}\Delta) \\
		& + D^TX^{-1}\Delta D^TGHGDD^TX^{-1}\Delta.
	\end{align*}
	Since $D^TGHGD - D^THD = \Delta^THD + D^TH\Delta + \Delta^T H \Delta$, the operator norm of the third term is bounded above by $\rho_6 r_k \|H\|$ for some $\rho_6$. Adding together the three terms, there is a constant $\kappa_2$ with $\|H_{k+1} - \wt{H}\|_F \leq (\|H_k - I\|_F + 1)\kappa_2 r_k$.
\end{proof}
Since superlinear convergence is an asymptotic property, we may assume $k_0 = 1$ in Lemma~\ref{GROWTH}. We will also need the following technical result from \cite{DM1}.
\begin{LMA}[3.3 of \cite{DM1}]\label{DM_C}
	Let $\{\nu_k\}$ and $\{\delta_k\}$ be sequences of non-negative numbers such that $\nu_{k+1} \leq (1 + \delta_k)\nu_k + \delta_k$ and $\sum_{k=1}^\infty \delta_k < \infty$. Then $\{\nu_k\}$ converges.
\end{LMA}

\begin{CORO}\label{CONDITIONBD}
	$\{\|B_k - I\|_F \}_{k=1}^\infty$ and $\{\|H_k - I\|_F \}_{k=1}^\infty$ converge, and are therefore uniformly bounded. As an immediate corollary, $\{\|B_k\|_F\}_{k=1}^\infty$ and $\{\|H_k\|_F\}_{k=1}^\infty$ are also uniformly bounded.
\end{CORO}
\begin{proof}
	By \Cref{IMPR} and \Cref{GROWTH}, we have
	$$ \|H_{k+1} - I\|_F \leq \|H_{k+1} - \wt{H}_{k+1}\|_F + \|\wt{H}_{k+1} - I\|_F \leq (1 + \kappa_2 r_k)\|H_k - I\|_F + \kappa_2 r_k$$
	Set $\nu_k = \|H_k - I\|_F$ and $\delta_k = \kappa_2 r_k$ in \Cref{DM_C}. Since $\sum_{k=1}^\infty r_k < \infty$, the sequence $\{\|H_k - I\|_F\}$ converges. The same reasoning applies to $\{\|B_k - I\|_F\}$.
\end{proof}

\begin{LMA}\label{psiphiLim}
	Recall the notation introduced in \Cref{B_IMPR}: $P_k$ is the orthogonal projection onto $\opn{Col}(B_k^\half D_k)$, and $P_{B_k}$ the orthogonal projection onto $\opn{Col}(B_kD_k)$. Define the quantities $\varphi_k, \psi_k$ to be
	\begin{align*}
		\varphi_k &= \|P_{B_k} - B_k^\half P_k B_k^\half \|_F^2  \\
		\psi_k &= \opn{Tr}(B_k( B_k^\half P_k B_k^\half) - (B_k^\half P_k B_k^\half)^2)
	\end{align*}
	Then $\lim\limits_{k \rightarrow \infty} \varphi_k= 0$ and $\lim\limits_{k \rightarrow \infty} \psi_k = 0$.
\end{LMA}
\begin{proof}
	We first bound $\|\wt{B}_{k+1} - I\|_F^2$ in terms of $\|B_{k+1} - I\|_F^2$. By \Cref{GROWTH}, $\|B_{k+1} - \wt{B}_{k+1}\|_F \leq \kappa_1 r_k$. Let $\kappa_3 = 2 \kappa_1 \max\limits_k \{\|B_k - I\|_F\}$; by \Cref{CONDITIONBD}, the maximum exists. Using the triangle inequality, we have
	\begin{align*}
		\|\wt{B}_{k+1} - I\|_F^2 &\geq (\|B_{k+1} - I\|_F - \|B_{k+1} - \wt{B}_{k+1}\|_F)^2 \\
		&= \|B_{k+1} - I\|_F^2 - 2\|B_{k+1} - I\|_F\|B_{k+1} - \wt{B}_{k+1}\|_F + \|B_{k+1} - \wt{B}_{k+1}\|_F^2 \\
		&\geq \|B_{k+1} - I\|_F^2 - \kappa_3 r_k.
	\end{align*}
	By \Cref{B_IMPR}, $\|B_k - I\|_F^2 - \|\wt{B}_{k+1} - I\|_F^2 \geq 0$. Summing over $k$ and telescoping, we find that
	\begin{align*}
		\sum_{k=1}^\infty \left( \|B_k - I\|_F^2 - \|\wt{B}_{k+1} - I\|_F^2 \right) &\leq \sum_{k=1}^\infty \left( \|B_k - I\|_F^2 - \|B_{k+1} - I\|_F^2 \right) + \kappa_3 r_{k+1} \\
		&\leq \|B_1 - I\|_F^2 + \kappa_3 \sum_{k=1}^\infty r_{k+1} < \infty
	\end{align*}
	from which we deduce that $\|B_k - I\|_F^2 - \|\wt{B}_{k+1} - I\|_F^2 \rightarrow 0$. Expressed in terms of $\varphi_k$ and $\psi_k$, \Cref{B_IMPR} states that $\|B_k - I\|_F^2 - \|\wt{B}_{k+1} - I\|_F^2 = \varphi_k + 2\psi_k$ and $\varphi_k, \psi_k \geq 0$. Hence $\varphi_k, \psi_k$ converge to 0.
\end{proof}

Note that \Cref{psiphiLim} does not imply that $\|B_k - I\|_F \rightarrow 0$, since it is possible for $\limsup \|\wt{B}_{k+1} - I\|_F > 0$. It is well-known that for the classical BFGS method, the Hessian approximation $B_k$ might not converge to the Hessian at the optimal solution.

\begin{LMA}\label{LIMITS}
	For any $w_k \in \opn{Col}(D_k)$,
	$$\left( 1 - \frac{ w_k^TB_k^2w_k}{w_k^TB_kw_k} \right)^2 \leq \varphi_k \hspace{1em} \text{ and } \hspace{1em} 0 \leq \frac{w_k^TB_k^3w_k}{w_k^TB_kw_k} - \left( \frac{w_k^TB_k^2w_k}{w_k^TB_kw_k} \right)^2 \leq \varphi_k + \psi_k,$$
	where $\varphi_k, \psi_k$ are defined in \Cref{psiphiLim}. Consequently, for any sequence $\{w_k\}_{k=1}^\infty$ with $w_k \in \opn{Col}(D_k)$, we have $\lim\limits_{k \rightarrow \infty} \frac{ w_k^TB_k^2w_k}{w_k^TB_kw_k} = 1$ and $\lim\limits_{k \rightarrow \infty} \frac{w_k^TB_k^3w_k}{w_k^TB_kw_k} = 1$.
\end{LMA}
\begin{proof}
	For a fixed $k$, let $B = B_k, D = D_k$, and let $\Delta = (D^TB^2D)^{-1} - (D^TBD)^{-1}$. Recall the definitions of $P, P_B$ from \Cref{B_IMPR}. We can write
	\begin{align*}
		\varphi_k = \|P_B - B^\half P B^\half\|_F^2 = \opn{Tr}((BD\Delta D^TB)^2) &= \opn{Tr}(D^TB^2D \Delta D^TB^2D \Delta)\\
		&= \opn{Tr}((I - D^TB^2D(D^TBD)^{-1})^2)
	\end{align*}
	Take a $B_k$-orthogonal basis $\{v_1,\ldots,v_{q_k}\}$ for $\opn{Col}(D_k)$ with $v_1 = w_k$. The $i$-th diagonal entry of $(I - D^TB^2D(D^TBD)^{-1})^2$ is then
	$$\left( 1 - \frac{ v_i^TB^2v_i}{v_i^TB v_i} \right)^2 + \sum_{j \neq i} \frac{ (v_i^TB^2v_j)^2}{v_i^TBv_i v_j^TBv_j}$$
	Since every term is non-negative, we conclude that $\left( 1 - \frac{ w_k^TB^2w_k}{w_k^TBw_k} \right)^2 \leq \varphi_k$, which proves the first statement. Also, notice that $\sum_{i=1}^{q_k} \sum_{j \neq i} \frac{ (v_i^TB^2v_j)^2}{v_i^TBv_i v_j^TBv_j} \leq \varphi_k$.
	
	Next, write $\opn{Tr}(B(B^\half P B^\half)) = \opn{Tr}(D^TB^3D(D^TBD)^{-1})$ and  $\opn{Tr}((B^\half P B^\half)^2) = \opn{Tr}((D^TB^2D(D^TBD)^{-1})^2)$. Again taking a $B_k-$orthogonal basis $\{v_1,\ldots,v_{q_k}\}$, we have
	\begin{align*}
		\opn{Tr}(D^TB^3D(D^TBD)^{-1}) &= \sum_{i=1}^{q_k} \frac{v_i^TB^3v_i}{v_i^TBv_i}\\
		\opn{Tr}((D^TB^2D(D^TBD)^{-1})^2) &= \sum_{i=1}^{q_k} \left( \frac{v_i^TB^2v_i}{v_i^TBv_i} \right)^2 + \sum_{i=1}^{q_k} \sum_{j \neq i} \frac{ (v_i^TB^2v_j)^2}{v_i^TBv_i v_j^TBv_j}
	\end{align*}
	Thus
	\begin{align*}
		\opn{Tr}(B(B^\half P B^\half) - (B^\half P B^\half)^2) &= \sum_{i=1}^{q_k} \left( \frac{v_i^TB^3v_i}{v_i^TBv_i} - \left( \frac{v_i^TB^2v_i}{v_i^TBv_i} \right)^2 \right) - \sum_{i=1}^{q_k} \sum_{j \neq i} \frac{ (v_i^TB^2v_j)^2}{v_i^TBv_i v_j^TBv_j} \\
		&\geq \sum_{i=1}^{q_k} \left( \frac{v_i^TB^3v_i}{v_i^TBv_i} - \left( \frac{v_i^TB^2v_i}{v_i^TBv_i} \right)^2 \right) - \varphi_k
	\end{align*}
	By the Cauchy-Schwarz inequality applied to $v^TB^2v = \la B^\half v, B^\frac{3}{2}v\ra$, we have $\frac{v^TB^3v}{v^TBv} \geq \left(\frac{v^TB^2v}{v^TBv} \right)^2$ for every $v \in \RR^n$. Hence $0 \leq \frac{w_k^TB^3w_k}{w_k^TBw_k} - \left(\frac{w_k^TB^2w_k}{w_k^TBw_k} \right)^2 \leq \varphi_k + \psi_k$. The limits then follow from \Cref{psiphiLim}, since $\varphi_k, \psi_k \rightarrow 0$.
\end{proof}

\begin{CORO}\label{DM_SUP}
	Given any $w_k \in \opn{Col}(D_k)$,
	$$\frac{\|(B_k - I) w_k \|}{ \|w_k\|} \leq \sqrt{2\varphi_k + \psi_k}$$
	Consequently, for any sequence $\{w_k\}_{k=1}^\infty$ with $w_k \in \opn{Col}(D_k)$,
	$$\lim_{k \rightarrow \infty} \frac{\|(B_k - I)w_k\|}{\|w_k\|} = 0$$
\end{CORO}
\begin{proof}
	By \Cref{LIMITS} and a routine calculation,
	\begin{align*}
		\frac{ \| B_k^\frac{1}{2} (B_k - I) w_k \|}{ \|B_k^\frac{1}{2}w_k\|} &= \sqrt{ \frac{ w_k^TB_k^3w_k}{w_k^TB_kw_k} - 2\frac{w_k^TB_k^2w_k}{w_k^TB_kw_k} + 1} \\
		&= \sqrt{ \frac{ w_k^TB_k^3w_k}{w_k^TB_kw_k} - \left( \frac{w_k^TB_k^2w_k}{w_k^TB_kw_k} \right)^2 + \left( 1 - \frac{w_k^TB_k^2w_k}{w_k^TB_kw_k} \right)^2 } \\
		&\leq \sqrt{2\varphi_k + \psi_k}
	\end{align*}
	Since $\{\|B_k\|\}, \{\|H_k\|\}$ are uniformly bounded by \Cref{CONDITIONBD}, the result follows.
\end{proof}

\begin{LMA}\label{STEPSIZE_1}
	A step size of $\lambda_k = 1$ is eventually admissible for steps $d_k$ included in $D_k$.
\end{LMA}
\begin{proof}
	We check that $\lambda_k = 1$ satisfies the Armijo-Wolfe conditions for all sufficiently large $k$. Let $\alpha$ and $\beta$ be the Armijo-Wolfe parameters and choose a constant $\gamma$ such that $0 < \gamma < \frac{\half - \alpha}{1 - \alpha}$. By \Cref{DM_SUP}, for all sufficiently large $k$, the steps $d_k \in \opn{Col}(D_k)$ satisfy
	\begin{equation}\label{eq:gammaBd}\frac{\|(B_k - I) d_k \|}{ \|d_k\|} \leq \gamma
	\end{equation}
	in which case $\la g_k, d_k \ra = \la g_k + d_k, d_k \ra - \|d_k\|^2 \leq -(1 - \gamma) \|d_k\|^2$.
	
	By Taylor's theorem, there exists a point $\wt{x}_k$ on the line segment joining $x_k, x_k + d_k$ with $f(x_k + d_k) = f(x_k) + \la g_k, d_k\ra + \half d_k^TG(\wt{x}_k) d_k$. Since $f(x_k) \leq f(\ux{x}{k-1}{q+1})$, the strong convexity of $f$ implies that $\|x_k - x_\ast\| \leq \sqrt{M/m}~r_{k-1}$. Hence, taking $\rho_7 = \mu \sqrt{M/m}$, we have $\|G(\wt{x}_k) - I\| \leq \mu \|\wt{x}_k - x_\ast\| \leq \rho_7 (r_{k-1} + \|d_k\|)$. For the step size $\lambda_k = 1$,
	\begin{align*}
		f(x_k + d_k) - f(x_k) &= \alpha \la g_k, d_k\ra + (1 - \alpha) \la g_k, d_k \ra +  \half d_k^TG(\wt{x})d_k \\
		&\leq \alpha \la g_k, d_k\ra - \left((1-\alpha)(1-\gamma) -  1/2 -  (\rho_7/2) (r_{k-1} + \|d_k\|) \right) \|d_k\|^2
	\end{align*}
	Since $(1-\alpha)(1-\gamma) - 1/2 > 0$ and $r_{k-1} + \|d_k\| \rightarrow 0$, a step size of $\lambda_k = 1$ satisfies the Armijo condition (\ref{armijo}) for all sufficiently large $k$.
	
	Next, apply Taylor's theorem to the function $t \mapsto \la g(x_k + td_k), d_k \ra$ to obtain a point $\wt{x}_k$ on the line segment joining $x_k, x_k + d_k$ with $\la g(x_k + d_k), d_k \ra = \la g_k, d_k \ra + d_k^TG(\wt{x}_k)d_k$. Choosing $\gamma = \frac{\beta}{2 - \beta}$ in (\ref{eq:gammaBd}), \Cref{DM_SUP} implies that for sufficiently large $k$, $\la -g_k, d_k \ra = \la g_k + d_k, -d_k \ra + \|d_k\|^2 \leq (1-\half \beta)^{-1} \|d_k\|^2$. We can also take $k$ large enough so that $1 - \rho_7 (r_{k-1} + \|d_k\|) \geq 0$, and we then have
	\begin{align*}
		\la g(x_k + d_k),d_k \ra &\geq \la g_k, d_k  \ra + (1 - \rho_7 (r_{k-1} + \|d_k\|))\|d_k\|^2 \\
		&\geq (\beta/2 + ( 1 - \beta/2) \rho_7 (r_{k-1} + \|d_k\|) ) \la g_k, d_k \ra
	\end{align*}
	Thus, the Wolfe condition (\ref{wolfe}) is satisfied for all sufficiently large $k$.
\end{proof}

\Cref{STEPSIZE_1} applies only to steps $d_k$ included in $D_k$. However, since Block BFGS does not prefer any particular step for inclusion in $D_k$, it is likely that eventually $\lambda_k = 1$ is admissible for \emph{all} steps. This issue reveals a subtle artifact of the proof method, and we return to discuss it in the remark after the following proof of \Cref{SUPERLINEAR}.

\begin{proof}(of \Cref{SUPERLINEAR})
	Let $\ux{s}{k}{i}$ be any step included in $D_k$. To simplify the notation, we write $x = \ux{x}{k}{i}, x^+ = \ux{x}{k}{i+1}, g = \ux{g}{k}{i}, g^+ = \ux{g}{k}{i+1}$, and $d = \ux{d}{k}{i}, s = \ux{s}{k}{i}$. By Lemma~\ref{STEPSIZE_1}, eventually $\lambda = 1$ is admissible for all steps in $D_k$, so $s = d$. From the triangle inequality, $\|d\| \leq \|x - x_\ast\| + \|x^+ - x_\ast\|$, so
	\begin{equation}\label{eq:g_over_r}
		\frac{\|g^+\|}{\|d\|} \geq \frac{m\|x^+ - x_\ast\|}{\|x - x_\ast\| + \|x^+ - x_\ast\|}.
	\end{equation}
	Next, write
	\begin{align*}\frac{\|(B_k - I)d\|}{\|d\|} &= \frac{\|g(x + d) - g(x) - G(x_\ast)d - g(x + d)\|}{\|d\|} \\
		&\geq \frac{\|g(x + d)\|}{\|d\|} - \frac{\|g(x + d) - g(x) - G(x_\ast)d\|}{\|d\|}.
	\end{align*}
	By continuity of the Hessian, the second term converges to 0. Thus, Corollary~\ref{DM_SUP} implies that $\frac{\|g^+\|}{\|d\|} = \frac{\|g(x + d)\|}{\|d\|} \rightarrow 0$. We deduce from (\ref{eq:g_over_r}) that
	$$\frac{\|x^+ - x_\ast\|}{\|x - x_\ast\|} \rightarrow 0.$$
	Hence, we have $Q$-superlinear convergence along the subsequence of steps in $D_k$.
\end{proof}

The same argument, with minimal alteration, applies to Rolling Block BFGS.

\subsection*{Remarks}
\begin{enumerate}
	\item As we observed earlier, the choice to include $\ux{s}{k}{1}$ in $D_k$ is arbitrary. The proof of Theorem~\ref{SUPERLINEAR} holds with \emph{any} selection rule for $D_k$ as long as it guarantees $\sum_{k=1}^\infty r_k < \infty$. Therefore, it is likely that Theorem~\ref{SUPERLINEAR} and Lemma~\ref{STEPSIZE_1} apply to \emph{all} steps. That is, eventually $\lambda_k = 1$ is admissible for all steps and $\frac{ \|\ux{x}{k}{i+1} - x_\ast\|}{\|\ux{x}{k}{i} - x_\ast\|} \rightarrow 0$. In fact, by selecting $D_k$ in a particular way, we can ensure that eventually $\lambda_k = 1$ is admissible for all steps. 
	
	\begin{CORO}\label{ADVERSARY1}
		Suppose that $D_k$ is constructed to always contain a step for which $\lambda_k = 1$ is not admissible, whenever such a step exists in the $k$-th block. Then $\lambda_k = 1$ is eventually admissible for \emph{all} steps.
	\end{CORO}
	\begin{proof}
		When executing the $k$-th update, we specifically set the first column of $D_k$ to a step $d_k$ from the $k$-th block for which $\lambda_k = 1$ is not admissible, if any such step exists. If we could find such a step $d_k$ for infinitely many $k$, then this process would produce an infinite sequence of steps $d_k \in \opn{Col}(D_k)$ for which $\lambda_k = 1$ is never eventually admissible. This contradicts Lemma~\ref{STEPSIZE_1}. \end{proof}
	However, \Cref{ADVERSARY1} does \emph{not} show that in general, $\lambda_k = 1$ is eventually admissible for all steps, as it only holds when we select steps in an adversarial manner. This example highlights an interesting dichotomy arising from our proof method. On one hand, \Cref{SUPERLINEAR} and \Cref{STEPSIZE_1} are retrospective and apply to any sequence $\{D_k\}$ that we select. This strongly suggests that they should hold for all steps. On the other hand, the method of proof (based on analyzing the convergence of $\|B_k - I\|_F^2 - \|\wt{B}_{k+1} - I\|_F^2$) makes use only of the steps in $D_k$, and thus can only prove things about the steps in $D_k$.
	
	\item The parameter $\tau$ has no equivalent in the classical BFGS method, and enforces a lower bound on the curvature of steps used in the update. If $\tau$ is chosen to be too large, then it is possible that $B_k$ is not updated on some iterations; in this case, the convergence rate will not be superlinear. A \emph{sufficient} condition for $B_k$ to be updated on every iteration, and hence for superlinear convergence, is to take $\tau \leq m$, but this requires knowledge of a lower bound on $m$, the least eigenvalue of the Hessian.
	
	This issue can be avoided if $f$ is strongly convex on the entire level set $\Omega = \{x \in \RR^n: f(x) \leq f(x_1)\}$, by using a slightly modified version of \textsc{filtersteps}. Instead of $\tau$, the user selects any $\tau' > 0$. The first step $\ux{s}{k}{1}$ is unconditionally included in $D_k$, and then subsequent steps $\ux{s}{k}{2},\ldots,\ux{s}{k}{q+1}$ are included only if the condition $\sigma_i^2 > \tau'\|\ux{s}{k}{i}\|^2$ holds. Since $\sigma_1^2 = \la \ux{s}{k}{1}, G_k\ux{s}{k}{1} \ra \geq m \|\ux{s}{k}{1}\|^2$, every entry of $\Sigma$ satisfies $\sigma_i \geq \tau \|\ux{s}{k}{i}\|^2$ for $\tau = \min\{\tau', m\} > 0$, and thus the condition for convergence is satisfied. This guarantees $Q$-superlinear convergence for any choice of $\tau'$, although larger $\tau'$ reduces the number of steps in $D_k$ (see \Cref{SUPERLINEAR}).

\end{enumerate}
\section{Modified Block BFGS for Non-Convex Optimization}\label{sec:NCONV}

Convergence theory for the classical BFGS method does not extend to non-convex functions. However, with minor modifications, BFGS performs well for non-convex optimization and can be shown to converge in some cases. Modifications that have been studied include:
\begin{enumerate}
	\item\label{itm:lifu_cautious} Cautious Updates (Li and Fukushima, \cite{LK_SIAM}) \\
	A BFGS update is performed only if $ \frac{y_k^Ts_k}{\|s_k\|^2} \geq \epsilon \|g_k\|^\alpha$, where $\epsilon, \alpha$ are parameters.
	\item\label{itm:lifu_mod} Modified Updates (Li and Fukushima, \cite{LK_JCAM}) \\
	The secant equation is modified to $B_{k+1}s_k = z_k$, where $z_k = y_k + r_k s_k$ and the parameter $r_k$ is chosen so that $z_k^Ts_k \geq \epsilon \|s_k\|^2$.
	\item Damped BFGS (Powell, \cite{POW_DAMP})\\
	The secant equation is modified to $B_{k+1}s_k = z_k$, where $z_k = \theta_k y_k + (1-\theta_k)B_ks_k$, and for $0 < \phi < 1$, the damping constant $\theta_k$ is determined by
	$$\theta_k = \left\{ \begin{array}{cl} 1, & \text{ if } y_k^Ts_k \geq \phi s_k^TB_ks_k \\ \frac{(1-\phi)s_k^TB_ks_k}{s_k^TB_ks_k - y_k^Ts_k}, & \text{ otherwise} \end{array} \right.$$
	This is perhaps the most widely used modified BFGS method. Unfortunately, no global convergence proof is known for this method.
\end{enumerate}

We show Block BFGS converges for non-convex functions, and describe analogous modifications for block updates. The next theorem provides a framework for proving convergence in the non-convex setting.

\begin{THM}\label{GEN_NCVX_CONV}
	Assume $f$ is twice differentiable and $-MI \preceq G(x) \preceq MI$ for all $x$ in the convex hull of the level set $\{x \in \RR^n: f(x) \leq f(x_1)\}$. Suppose that $\{\wt{G}_k\}_{k=1}^\infty$ is a sequence of symmetric matrices satisfying, for all $k$, the conditions
	\begin{enumerate}
		\item $-MI \preceq \wt{G}_k \preceq MI$
		\item For some constant $\eta > 0$, the matrix $D_k$ produced by $\textsc{filtersteps}(S_k, \wt{G}_k)$ satisfies $D_k^T\wt{G}_kD_k \succeq \eta D_k^TD_k$
	\end{enumerate}
	Then we may perform Block BFGS using the updates
	$$B_{k+1} = B_k - B_kD_k(D_k^TB_kD_k)^{-1}D_k^TB_k + \wt{G}_kD_k(D_k^T\wt{G}_kD_k)^{-1}D_k^T\wt{G}_k$$
	and Block BFGS converges in the sense that $\liminf_k \|g_k\| = 0$.
\end{THM}

\begin{proof}
	The proof follows that of Theorem~\ref{CONVERGENT}, with several changes. First, note that Lemma~\ref{BPosDef} implies that $B_{k+1}$ remains positive definite, since \textsc{filtersteps} ensures that $D_k^T\wt{G}_kD_k$ is positive definite. Observe that Lemma~\ref{LIPSCH_KEY} continues to hold, as the condition $-MI \preceq G(x) \preceq MI$ for all $x$ in the convex hull of the level set implies that the gradient $g$ is Lipschitz with constant $M$. In Lemma~\ref{TRACE}, take the constant $c_3$ to  be $c_3 = \opn{Tr}(B_1) + \frac{qM^2}{\eta}$ and notice that
	$$\opn{Tr}(\wt{G}_jD_j(D_j^T\wt{G}_jD_j)^{-1}D_j^T\wt{G}_j) \leq \frac{1}{\eta} \opn{Tr}(\wt{G}_j D_j(D_j^TD_j)^{-1}D_j^T \wt{G}_j) \leq \frac{qM^2}{\eta}$$
	where the last inequality follows because $D_j(D_j^TD_j)^{-1}D_j^T$ is the orthogonal projection onto $\opn{Col}(D_j)$ and has rank $q_j \leq q$, and $\|\wt{G}_j D_j(D_j^TD_j)^{-1}D_j^T \wt{G}_j\| \leq \|\wt{G}_j\|^2 = M^2$.
	
	The remainder of the proof is similar to Theorem~\ref{CONVERGENT}.
\end{proof}

\begin{LMA}\label{BB_NCV_CONV}
	Assume $f$ is twice differentiable and $-MI \preceq G(x) \preceq MI$ for all $x$ in the level set $\{x \in \RR^n: f(x) \leq f(x_1)\}$. If $D_k^TG_kD_k$ satisfies $\sigma_i^2 \geq \tau \|s_i\|^2$, where $\sigma_i$ is the $i$-th diagonal entry of the $L\Sigma L^T$ decomposition of $D_k^TG_kD_k$, then $D_k^TG_kD_k \succeq \eta D_k^TD_k$ for $\eta = \frac{\tau^q}{q^qM^{q-1}}$.
\end{LMA}
\begin{proof}
	Let $G = G_k, D = D_k$. Without loss of generality, we may assume the columns of $D$ have norm 1, as otherwise we can normalize $D$ by right-multiplying by a positive diagonal matrix. Then the diagonal entries $\sigma_i^2$ of the $L \Sigma L^T$ decomposition of $D^TGD$ satisfy $\sigma_i^2 \geq \tau$.
	
	Order the eigenvalues of $D^TGD$ as $\lambda_1 \geq \lambda_2 \geq \ldots \geq \lambda_q > 0$. We have
	$$\lambda_q = \frac{\opn{det}(D^TGD)}{\prod_{i=1}^{q-1} \lambda_i } \geq \frac{ \tau^q}{(qM)^{q-1}}.$$
	Since every column of $D$ has norm 1, the eigenvalues of $D^TD$ are bounded by $\opn{Tr}(D^TD) = q$. Hence $I \succeq \frac{1}{q} D^TD$ and so $ D^TGD \succeq \frac{\tau^q}{(qM)^{q-1}} I \succeq \frac{\tau^{q}}{q^qM^{q-1}} D^TD$.
\end{proof}

Block BFGS (Algorithm~\ref{alg:BLOCK}) satisfies the conditions of Lemma~\ref{BB_NCV_CONV} when we take $\wt{G}_k = G_k$ and apply \textsc{filtersteps} (Algorithm~\ref{alg:FILTER}). Thus Theorem~\ref{GEN_NCVX_CONV} shows that Block BFGS converges globally for non-convex functions. The filtering procedure is analogous to the cautious update (\ref{itm:lifu_cautious}) of Li and Fukushima, and hence, it is possible, although very unlikely, that filtering will produce an empty $D_k$.  Hessian modification and Powell's damping method can also be extended to block updates.

\section{Numerical Experiments}\label{sec:NUMER}
We evaluate the performance of several block quasi-Newton methods by generating a \emph{performance profile} \cite{DOLMOR}, which can be described as follows. Given a set of algorithms $\mathcal{S}$ and a set of problems $\mathcal{P}$, let $t_{s,p}$ be the cost for algorithm $s$ to solve problem $p$. For each problem $p$, let $m_p$ be the minimum cost to solve $p$ of any algorithm. A performance profile is a plot comparing the functions
$$\rho_s(r) = \frac{ | \{ p \in \mathcal{P}: t_{s,p}/m_p \leq r \}|}{|\mathcal{P}|} $$
for all $s \in \mathcal{S}$. Observe that $\rho_s(r)$ is the fraction of problems in $\mathcal{P}$ that algorithm $s$ solved within a factor $r$ of the cost of the best algorithm for problem $p$. As reference points, we include the classical BFGS method and gradient descent in $\mathcal{S}$.

For our inexact line search, we used the function \texttt{WolfeLineSearch} from \emph{minFunc} \cite{MINFUNC}, a mature and widely used Matlab library for unconstrained optimization. The line search parameters were $\alpha = 0.1$ and $\beta = 0.75$, and \texttt{WolfeLineSearch} was configured to use interpolation with an initial step size $\lambda = 1$ (options $\texttt{LS\_type} = 1, \texttt{LS\_init} = 0, \texttt{LS\_interp} = 1, \texttt{LS\_multi} = 0$).

From preliminary experiments, we found that large values of $q$ tend to increase numerical errors, eventually leading to search directions $d_k$ that are not descent directions. This effect is particularly pronounced when $q \geq \sqrt{n}$. The experiments in \cite{GGR} also obtained the best performance when $\floor{n^{1/4}} \leq q \leq \sqrt{n}$. In creating performance profiles, we opted for $q = \floor{n^{1/3}}$. 

\subsection{Convex Experiments}

We compared the methods listed below. 
\begin{enumerate}
	\item \emph{BFGS}
	\item \emph{Block BFGS Variant 1}, or \emph{B-BFGS1}
	
	Block BFGS (Algorithm~\ref{alg:BLOCK}). We store the full inverse Hessian approximation $H_k$ and compute $d_k = -H_kg_k$ by a matrix-vector product. We do not perform \textsc{filtersteps}, so the update (\ref{eq:BFGS_INV_UPD}) uses all steps.
	
	\item \emph{Block BFGS Variant 2}, or \emph{B-BFGS2}
	
	Block BFGS (Algorithm~\ref{alg:BLOCK}), with Algorithm~\ref{alg:FILTER} and $\tau =  10^{-3}$. As in B-BFGS1, the full Hessian approximation $H_k$ is stored. $H_k$ is updated by (\ref{eq:BFGS_INV_UPD}) using the steps returned by Algorithm~\ref{alg:FILTER}.
	
	\item \emph{Block BFGS with $q = 1$}, or \emph{B-BFGS-q1}
	
	This compares the effect of using a single sketching equation as in Block BFGS updates versus using the standard secant equation of BFGS updates.
	
	\item \emph{Rolling Block BFGS}, or \emph{RB-BFGS}
	
	See Section~\ref{sub:RBBFGS}. We take a smaller value $q = \min \{3, \floor{n^{1/3}}\}$ for this method, and omit filtering.
	
	\item \emph{Gradient Descent}, or \emph{GD}
\end{enumerate}

Each algorithm is considered to have \emph{solved} a problem when it reduces the objective value to less than some threshold $f_{stop}$. The thresholds $f_{stop}$ are pre-computed for each problem $p$ by minimizing $p$ with minFunc to obtain a near-optimal solution $f_\ast$, and setting $f_{stop} = f_\ast + 0.01|f_\ast|$.

We measure the cost $t_{s,p}$ in two metrics: the number of steps, and the amount of CPU time. Every step $\ux{s}{k}{i}$ is counted once when measuring the number of steps.

\subsubsection{Logistic Regression Tests}\label{subsub:LOGREG}

As in \cite{GGR}, we ran tests on \emph{logistic regression} problems, a common classification technique in statistics. For our purposes, it suffices to describe the objective function. Given a set of $m$ data points $(y_i, x_i)$, where $y_i \in \{0,1\}$ is the class, and $x_i \in \RR^n$ is the vector of features of the $i$-th data point, we minimize, over all weights $w \in \RR^n$, the loss function
\begin{equation}\label{eq:logreg}
	L(w) = -\frac{1}{m}\sum_{i=1}^m \log \phi(y_i, x_i, w) + \frac{1}{2m} w^TQw
\end{equation}
$$\phi(y_i, x_i, w) = \left\{ \begin{array}{ll} \frac{1}{1 + \exp(-x_i^Tw)} & \text{ if } y_i = 1 \\ 1 - \frac{1}{1 + \exp(-x_i^Tw)} & \text{ if } y_i = 0 \end{array} \right.$$
where $Q \succ 0$ in the 'regularization' term. Figure~\ref{fig:LR} shows the performance profiles for this test. See Appendix~\ref{appdx:EXP} for a list of the data sets and our choices for $Q$.

\begin{figure}
	\centering 
	\includegraphics[clip=true, scale = \imgscale]{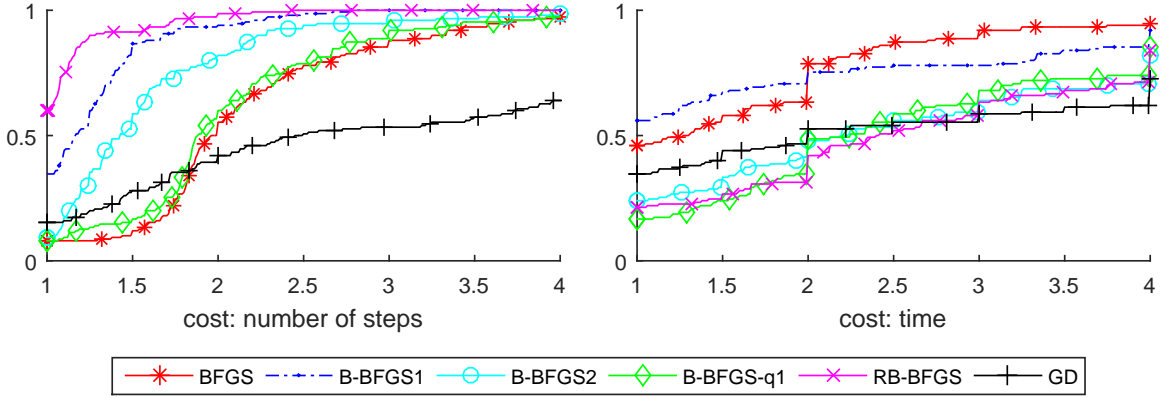}
	\caption{Logistic Regression profiles ($\rho_s(r)$)}
	\label{fig:LR}
\end{figure}

In Figure~\ref{fig:LR}, we see that the block methods B-BFGS1, B-BFGS2, and RB-BFGS all outperform BFGS in terms of the number of steps to completion. Considering the amount of CPU time used, B-BFGS1 is competitive with BFGS, while B-BFGS2 and RB-BFGS are more expensive than BFGS. This suggests that the additional curvature information added in block updates allows Block BFGS to find better search directions, but at the cost of the update operation being more expensive. B-BFGS-q1 and BFGS exhibit very similar performance when measured in steps, so there appears to be little difference between using a single sketching equation and a secant equation on this class of problems.

Interestingly, B-BFGS1 outperformed B-BFGS2, indicating that steps are being removed from the update, which would improve the search directions. The most likely explanation is that $\tau=10^{-3}$ is excessively large relative to the eigenvalues of $G(x)$.
\subsubsection{Log Barrier QP Tests}\label{subsub:LOGB}
We tested problems of the form
\begin{equation}\label{eq:logb}
	\min_{y \in \RR^s} ~ F(y) = \half y^T\ov{Q}y + \ov{c}^Ty - 1000 \sum_{i=1}^n \log (\ov{b}- \ov{A}y)_i
\end{equation}
where $\ov{Q} \succeq 0$, $\ov{c} \in \RR^s$, $\ov{b} \in \RR^n$, and $\ov{A} \in \RR^{n \times s}$. Note that the objective value is $+\infty$ if $y$ does not satisfy $\ov{A}y < \ov{b}$. In Appendix~\ref{appdx:EXP}, we explain how to derive a log barrier problem from a QP in standard form. See  Figure~\ref{fig:LB} for the performance profile. Note that problems with a barrier structure are atypical in the context of unconstrained minimization, and are usually solved with specific interior point methods. However, they are somewhat interesting as they can be quite challenging to solve.

\begin{figure}
	\centering
	\includegraphics[clip=true, scale=\imgscale]{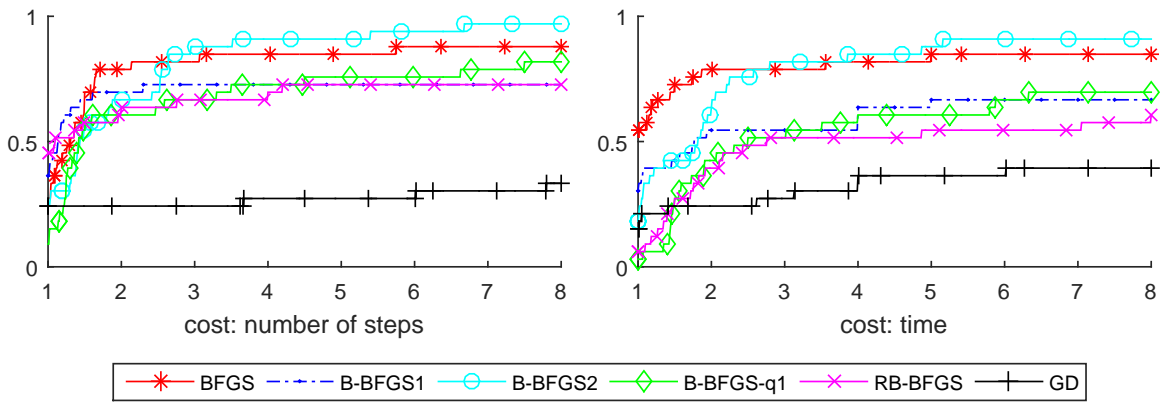}
	\caption{Log Barrier QP profiles ($\rho_s(r)$)}
	\label{fig:LB}
\end{figure}

Since $\nabla^2 F(y) = Q + 1000\ov{A}^T S \ov{A}$ where $S$ is diagonal with entries $(\ov{b} - \ov{A}y)_i^{-2}$, these problems are often extremely ill-conditioned. This leads to issues when using \texttt{WolfeLineSearch}, as the line search can require many backtracking iterations, or even fail completely, when the current point is near the boundary of the log barrier. This causes particular issues with block updates, as $\nabla^2 F(y)$ has small numerical rank when $S$ has a small number of extremely large entries. Consequently, we removed problems from the test set which were ill-conditioned to the extent that even after performing step filtering, the line search failed at some step before reaching the optimal solution. Quasi-Newton methods, and those using block updates with large $q$ in particular, are poorly suited for these ill-conditioned problems. However, although the standard BFGS method also can fail on these problems, it is more robust than block methods.

\subsection{Non-Convex Experiments}\label{sub:ncvx}
Since non-convex functions often have multiple stationary points, more complex behavior is possible than in the convex case. For instance, one algorithm may generally require more steps to converge, but may be taking advantage of additional information to help avoid spurious local minima.

Let $f_p$ denote the best objective value obtained for problem $p$ by any algorithm. To evaluate both the early and asymptotic performance of our algorithms, we generated performance profiles comparing the cost for each algorithm to reach a solution with objective value less than $f_p + \epsilon |f_p|$ for $\epsilon = 0.2$, $\epsilon = 0.1$, and $\epsilon = 0.01$. When $|f_p|$ is very small (for instance, $|f_p| < 10^{-10}$), we essentially have $f_p = 0$ and treat all solutions with objective value within $10^{-10}$ as being optimal.

We compared four different algorithms for non-convex minimization:
\begin{enumerate}
	\item \emph{Damped BFGS}, or \emph{D-BFGS}
	
	Damped BFGS with $\phi = 0.2$ (see Section~\ref{sec:NCONV}).
	
	\item \emph{Block BFGS}, or \emph{B-BFGS}
	
	Block BFGS (Algorithm~\ref{alg:BLOCK}) with $q = \floor{n^{1/3}}$ and $\tau = 10^{-5}$.
	
	\item \emph{Block BFGS with $q = 1$}, or \emph{B-BFGS-q1}
	
	Block BFGS (Algorithm~\ref{alg:BLOCK}) with $q = 1$ and $\tau = 10^{-5}$.
	
	\item \emph{Gradient Descent}, or \emph{GD}
\end{enumerate}

\subsubsection{Hyperbolic Tangent Loss Tests}\label{subsub:tanh}
This is also a classification technique; however, unlike the logistic regression problems in Section \ref{subsub:LOGREG}, these problems are generally non-convex.
Given a set of $m$ data points $(y_i, x_i)$ where $y_i \in \{0,1\}$ is the class, and $x_i \in \RR^n$ the features, we seek to minimize over $w \in \RR^n$ the loss function
$$ L(w) = \frac{1}{m}\sum_{i=1}^m \left( 1 - \tanh (y_i x_i^Tw) \right) + \frac{1}{2m}\|w\|^2$$
Figure~\ref{fig:tanh} presents performance profiles for $\epsilon = 0.2, 0.1, 0.01$, with cost measured in both steps and CPU time. See Appendix~\ref{appdx:EXP} for a list of the data sets.

\begin{figure}
	\centering
	\includegraphics[clip=true, scale=\imgscale]{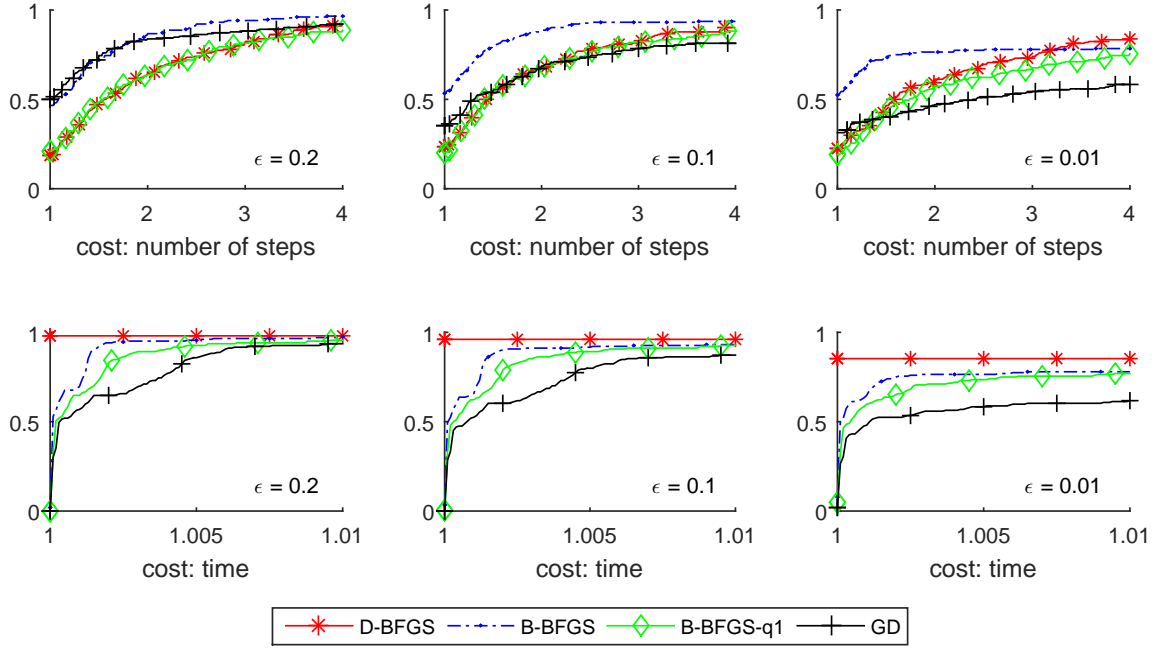}
	\caption{Hyperbolic Tangent Loss profiles ($\rho_s(r)$)}
	\label{fig:tanh}
\end{figure}

B-BFGS and gradient descent perform well at first, making rapid progress to within $0.2|f_p|$ of $f_p$ in the fewest number of steps. B-BFGS continues to converge quickly, generally requiring the fewest steps to reach $0.1|f_p|$ and $0.01|f_p|$ of $f_p$, while gradient descent is overtaken by BFGS and B-BFGS-q1.

Surprisingly, all four algorithms used nearly the same amount of CPU time, with each algorithm completing a majority of problems after using only 1\% more time than the fastest algorithm.

\subsubsection{Standard Benchmark Tests}\label{subsub:gen}
This test used 19 functions from the test collection of Andrei \cite{NATST}, many of which originate from the CUTEst test set. The functions are listed below, with the number of variables $n$ in parentheses: \\
\texttt{arwhead} (300), \texttt{bdqrtic} (200), \texttt{cube} (400), \texttt{diag1} (250), \texttt{dixonprice} (200), \texttt{edensch} (300), \texttt{eg2} (400), \texttt{explin2} (200), \texttt{fletchcr} (400), \texttt{genhumps} (250), \texttt{indef} (250), \texttt{mccormick} (400), \texttt{raydan1} (400), \texttt{rosenbrock} (300), \texttt{sine} (400), \texttt{sinquad} (400), \texttt{tointgss} (200), \texttt{trid} (200), \texttt{whiteholst} (300).

The gradients and Hessians were computed using the automatic differentiation program ADiGator \cite{AUTODIF}.

For each of these functions, we generated 6 random starting points and tested the 4 algorithms using each starting point, for a total of 114 problems. Figure~\ref{fig:gen} presents performance profiles for $\epsilon = 0.2, 0.1, 0.01$, with cost measured in steps. We see from Figure~\ref{fig:gen} that D-BFGS consistently outperforms B-BFGS-q1, which suggests that Powell's damping method is superior to cautious updates. 

\begin{figure}
	\centering
	\includegraphics[clip=true, scale=\imgscale]{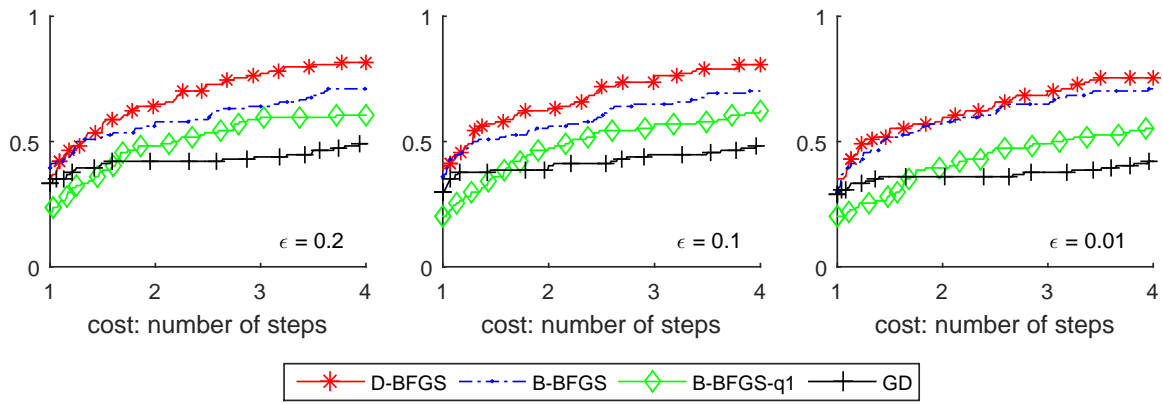}
	\caption{Standard Benchmark profiles ($\rho_s(r)$)}
	\label{fig:gen}
\end{figure}

\section{Concluding Remarks}

We have shown that Block BFGS provides the same theoretical rate of convergence as the classical BFGS method. Further investigation is needed to determine how Block BFGS performs on a wider range of real problems. In our experiments, we focused on a very basic implementation of Block BFGS, but many simple heuristics for improving performance and numerical stability are possible. In particular, it is important to select good values of $q$ and $\tau$ based on insights from the problem domain. We also briefly investigated the effect of using the action of the Hessian on the previous step versus the change in gradient over the previous step (as in classical BFGS) in constructing the update. Further study of the benefits and drawbacks of such an approach would be of interest, as would study of parallel implementation. We hope that this work will serve as a useful foundation for future research on quasi-Newton methods using block updates.

\section*{Acknowledgement}
\addcontentsline{toc}{Section}{Acknowledgement}
We would like to thank two anonymous referees for their helpful comments and suggestions.

\bibliographystyle{siam}
\raggedright	
\bibliography{blockbfgs}

\appendix
\section{Derivation of the Block BFGS Update Formula}\label{appdx:UPD_FORM}
Let $\|X\|_{G_k}$ denote the matrix norm $\opn{Tr}(XG_kX^TG_k)$. We show that the unique solution of
$$
(P) \hspace{1em} \left\{
\begin{array}{rl} \min\limits_{\wt{H} \in \RR^{n \times n}} &\|\wt{H} - H_k\|_{G_k} \\  \text{s.t } &\wt{H} = \wt{H}^T, \wt{H}G_kD_k = D_k\end{array}
\right.
$$
is given by formula (\ref{eq:BFGS_INV_UPD}). Introduce a new variable $E = \wt{H} - H_k$, and let $D = D_k$, $G = G_k$, $H = H_k, Y = G_kD_k, Z = D_k - H_kG_kD_k$. We rewrite the problem (P) in terms of $E$ and express its Lagrangian as
$$ \mathcal{L}(E, \Sigma, \Lambda) = \frac{1}{2}\opn{Tr}(EGE^TG) +\opn{Tr}(\Sigma(E - E^T)) + \opn{Tr}(\Lambda^T(EY - Z)$$

Solving $\frac{\partial \mathcal{L}}{\partial E} = 0$ in terms of $E$, we obtain $E = -G^{-1}(Y\Lambda^T + \Sigma - \Sigma^T)G^{-1}$. Thus $E - E^T = G^{-1}(\Lambda Y^T - Y\Lambda^T + 2(\Sigma^T - \Sigma))G^{-1} = 0$,
from which we obtain $\Sigma - \Sigma^T = \half(\Lambda Y^T - Y \Lambda^T)$. Therefore $E = -\half G^{-1}(Y\Lambda^T + \Lambda Y^T)G^{-1}$. 

To solve for $\Lambda$, substituting this expression for $E$ into the constraint $EY = Z$ yields
\begin{equation}\label{eq:LAG_AUX}
	G^{-1}(Y\Lambda^T + \Lambda Y^T)G^{-1}Y + 2Z = 0
\end{equation}
Left multiplying by $Y^T$ and using the definition $Y = GD$, we have
$$(D^TGD)(\Lambda^TD) + (D^T\Lambda)(D^TGD) + 2Y^TZ = 0$$
Now, it is easy to verify that $\Lambda^TD = -(D^TGD)^{-1}(Y^TZ)$ is the solution. Therefore, from (\ref{eq:LAG_AUX}), $\Lambda Y^TD = - Y\Lambda^T G^{-1}Y - 2GZ = Y(D^TGD)^{-1}Y^TZ - 2GZ$. Hence, $\Lambda = (Y(D^TGD)^{-1}Y^TZ - 2GZ)(D^TGD)^{-1}$. Substituting $\Lambda$ into our expression for $E$ and rearranging produces formula (\ref{eq:BFGS_INV_UPD}).

\section{Details of Experiments}\label{appdx:EXP}
\subsection{Logistic Regression Tests (\ref{subsub:LOGREG})}

The following 18 data sets from LIBSVM \cite{LIBSVM} were used:\\
\texttt{a1a}, \texttt{a2a}, \texttt{a3a}, \texttt{a4a}, \texttt{australian}, \texttt{colon-cancer}, \texttt{covtype}, \texttt{diabetes}, \texttt{duke}, \texttt{ionosphere-scale}, \texttt{madelon}, \texttt{mushrooms}, \texttt{sonar-scale}, \texttt{splice}, \texttt{svmguide3}, \texttt{w1a}, \texttt{w2a}, \texttt{w3a}.

Each data set was partitioned into 3 disjoint subsets with at most 2000 points. For each subset, we have a problem of the form (\ref{eq:logreg}) with the standard $L_2$ regularizer $Q = I$, producing 54 standard problems. An additional 96 problems with $Q = I + Q'$ were produced by adding a randomly generated convex quadratic $Q'$ to one of the standard problems. Two such problems were produced for each standard problem, except those from \texttt{duke} and \texttt{colon-cancer} (omitted for problem size).

\subsection{Log Barrier QP Tests (\ref{subsub:LOGB})}
Given a convex quadratic program $\min\limits_{x \in \RR^n} \{ \half x^TQx + c^Tx~|~ Ax = b, x \geq 0 \}$,  we derive a log barrier QP problem as follows. Taking a basis $N$ for the null space of $A$ (of dimension $s$), and a solution $Ax_0 = b, x_0 \geq 0$, the given QP is equivalent to $\min\limits_{y \in \RR^s} \{ \half y^T\ov{Q}y + \ov{c}^Ty~|~\ov{A}y \leq \ov{b}\}$, where $\ov{Q} = N^TQN, \ov{c} = N^T(c + Qx_0), \ov{b} = x_0$ and $\ov{A} = -N$. Replacing the constraint by a log barrier $-\mu \sum_{i=1}^n \log(\ov{b} - \ov{A}y)_i$ (with $\mu = 1000$), we obtain problem (\ref{eq:logb}).

This test included 43 problems in total. There were 35 log barrier problems derived from the QP test collection of Maros and M\'{e}sz\'{a}ros \cite{MMQP}:\\
\texttt{cvxqp1\_m}, \texttt{cvxqp1\_s}, \texttt{cvxqp2\_m}, \texttt{cvxqp2\_s}, \texttt{cvxqp3\_m}, \texttt{cvxqp3\_s}, \texttt{dual1}, \texttt{dual2}, \texttt{dual3}, \texttt{dual4}, \texttt{primal1}, \texttt{primal3}, \texttt{primal4}, \texttt{primalc1}, \texttt{primalc2}, \texttt{primalc5}, \texttt{primalc8}, \texttt{q25fv47}, \texttt{qbeaconf}, \texttt{qgrow15}, \texttt{qgrow22}, \texttt{qgrow7}, \texttt{qisrael}, \texttt{qscagr7}, \texttt{qscfxm1}, \texttt{qscfxm2}, \texttt{qscfxm3}, \texttt{qscorpio}, \texttt{qscrs8}, \texttt{qsctap1}, \texttt{qsctap3}, \texttt{qshare1b}, \texttt{qship08l}, \texttt{stadat1}, \texttt{stadat2}.

An additional 8 problems were derived from the following LP problems in the COAP collection \cite{COAP}: \texttt{adlittle}, \texttt{agg}, \texttt{agg2}, \texttt{agg3}, \texttt{bnl1}, \texttt{brandy}, \texttt{fffff800}, \texttt{ganges}.

\subsection{Hyperbolic Tangent Loss Tests (\ref{subsub:tanh})}

This test used the same data sets as the logistic regression test, with \texttt{duke} omitted because of large problem size ($n = 7130$). As in the logistic regression test, each data set was partitioned into 3 subsets with at most 2000 points, producing 51 loss functions. For each loss function, we tried 4 random starting points, for a total of 204 problems.
\end{document}